\documentclass{amsart}
\usepackage{amsmath, amssymb, amsfonts, url, color, hyperref}

\newtheorem{theorem}{Theorem}[section]
\newtheorem{lemma}{Lemma}[section]
\newtheorem{proposition}{Proposition}[section]
\newtheorem{corollary}{Corollary}[section]

\newtheorem*{question*}{Question}
\newtheorem*{problem*}{Problem}
\newtheorem{definition}{Definition}[section]
\newtheorem{example}{Example}[section]
\newtheorem{remark}{Remark}[section]
\newtheorem*{acknowledgements*}{Acknowledgements}
\numberwithin{equation}{section}
\newcommand{\CC}{\mathbb{C}}
\newcommand{\RR}{\mathbb{R}}

\begin{document}

\title[A new plurisubharmonic capacity]{A new plurisubharmonic capacity and functions holomorphic along holomorphic vector fields}

\author{Ye-Won Luke Cho}
\dedicatory{Dedicated to Professor Kang-Tae Kim on the occasion of his 65th birthday}

\begin{abstract}
    The main purpose of this article is to present a generalization of Forelli's 
	theorem for functions holomorphic along a suspension of integral curves of a diagonalizable vector field of aligned type. For this purpose, we develop a new capacity theory that generalizes the theory of projective capacity introduced by Siciak \cite{Siciak82}.  Our main theorem improves the results of \cite{KPS09}, \cite{Cho22} as well as the original Forelli's theorem.
\end{abstract}

\maketitle

 \section{Introduction}
 \subsection{Notations and terminology}
Let $X$ be a holomorphic vector field defined on an open neighborhood of the origin in $\mathbb{C}^n$. A vector field $X$ is said to be $\textit{contracting at the origin}$ if the flow-diffeomorphism $\Phi_t$ of $\textup{Re}\,X$ for some $t<0$ satisfies: (1) $\Phi_t(0)=0$, and (2) every eigenvalue of the matrix $d\Phi_t|_{0}$ has absolute value less than 1. By the Poincar\'{e}-Dulac theorem, there exists a local holomorphic coordinate system near the origin such that $X$ takes the following form:
\begin{equation}\label{contractingVF}
X=\sum_{j=1}^{n}(\lambda_jz_j+g_j(z))\frac{\partial}{\partial z_j},
\end{equation}
where $g_j\in \mathbb{C}[z_1,\dots,z_n]$ and $\lambda_j\in \mathbb{C}$ for each $j$. A vector field $X$ is said to be \textit{aligned} if $\lambda_j/\lambda_k>0$ for each $j,k\in \{1,\dots,n\}$. In this paper, we only consider $\textit{diagonalizable}$ vector fields of aligned type, i.e., the fields take the form (\ref{contractingVF}) with $g_j\equiv 0$ for each $j$. A vector field $X$ with eigenvalues $\lambda=(\lambda_1,\dots,\lambda_n)$ will be denoted as a pair $(X,\lambda)$. We will also assume without loss of generality that $\lambda_1=1,\;\lambda_k>0$ for each $k\in \{2,\dots,n\}$.

Denote by $B^n(a;r) := \{z \in \CC^n \colon \|z-a\|<r  \}$ and by
$S^m := \{v \in \RR^{m+1} \colon \|v\|=1\}$.  With such notation,
the boundary of $B^n:=B^n(0;1)$ is $S^{2n-1}$.  Recall that the complex flow map $\Phi^X(z,t)$ of a vector field $(X,\lambda)$ on $\mathbb{C}^n$ is given as 
\[
\Phi^X(z,t)=(z_1e^{-\lambda_1 t},\dots,z_ne^{-\lambda_n t}).
\]
\begin{definition}
\normalfont

Let $F\subset S^{2n-1}$ be a nonempty set and $\mathbb{H}$ the open right-half plane in $\mathbb{C}$ and consider
\[
S^X_0(F):=\{\Phi^{X}(z,t):z\in F,~t\in \mathbb{H}\}.
\]
By a $\textit{suspension of integral curves of X}$, we mean a pair of the form $(S^X_0(F),\Phi^{X})$. For simplicity, we will denote a suspension by its underlying set. Note that $S^X_0(F)$ is always $\lambda$-$\textit{balanced}$, i.e., $\Phi^{X}(z,t)\in S^X_0(F)$ for each $t\in \mathbb{H}$ and $z\in S^X_0(F)$. A suspension $S^X_0(F)$ is called a $\textit{formal Forelli suspension}$ if any function $f:B^n\to \mathbb{C}$ satisfying the following two conditions
\begin{enumerate}
	\setlength\itemsep{0.1em}
	\item $f\in C^{\infty}(0)$, i.e., for each positive integer $k$ there
	exists an open neighborhood $V_k$ of the origin $0$ such that $f \in C^k 
	(V_k)$, and
	\item $f$ is holomorphic along $S^X_0(F),$ i.e., $t\in \mathbb{H}\to f\circ \Phi^{X}(z,t)\in \mathbb{C}$ is holomorphic for each $z\in F$
\end{enumerate}
has a formal Taylor series $S_f=\sum a_{km}z^{k}\bar{z}^{m}$  of $\textit{holomorphic type}$, that is, $a_{km}=0$ whenever $m\neq 0$. See (\ref{formal Taylor series}) for the definition of $a_{km}$. We also say that $S^X_0(F)$ is a $\textit{normal suspension}$ if any formal power series $S\in \mathbb{C}[[z_1,\dots,z_n]]$ for which $S_z(t):=S\circ \Phi^{X}(z,t)$ is holomorphic in $t\in \mathbb{H}$ for every $z\in F$ converges uniformly on a neighborhood of the origin in $\mathbb{C}^n$. A formal Forelli suspension that is also normal is called a $\textit{Forelli suspension}$. 
\end{definition}

Note that any function $f:B^n\to \mathbb{C}$ that is smooth at the origin and holomorphic along a Forelli suspension is holomorphic on $B^n(0;r)$ for some $r>0$. To give a local characterization of Forelli suspensions, we introduce the following

\begin{definition}\label{def.susp.}
	\normalfont
	Fix $d_1,d_2\geq 0$ and let $(X,\lambda)$ be a vector field on $\mathbb{C}^n$. We say that $q\in \mathbb{C}[z_1,\dots,z_n,\bar{z}_1,\dots,\bar{z}_n]$ is $\textit{quasi-homogeneous of type}\; \lambda\;\textit{with bidegree}$ $(d_1,d_2)$  if 
	\[
	q(\Phi^X(z,t))=e^{-d_1t}e^{-d_2\bar{t}}q(z)
	\]
	for any $t\in \mathbb{H},\,z\in \mathbb{C}^n$. In this case, we use the notation $\textup{bideg}\,q=(d_1,d_2)$. We also denote by $\mathcal{H}_{\lambda}$ the set of all such polynomials. Let $\bar{F}$ be the closure of $F$ in $S^{2n-1}$. $S^X_0(F)$ is said to have an $\textit{algebraically}$ $\textit{nonsparse leaf}~ L_z:=\{\Phi^X(z,t):t\in \mathbb{H}\}$ generated by $z\in \bar{F}$ if the following is true: for each open neighborhood $U\subset S^{2n-1}$ of $z$ and $q\in \mathcal{H}_{\lambda}$ with $\textup{bideg}\,q=(d_1,d_2)$, $d_2\neq 0$, satisfying
	\[
	\bar{F}\cap U\subset Z(q):=\{z\in \mathbb{C}^n:q(z)=0\},
	\] 
	we have $q\equiv 0$ on $\mathbb{C}^n$. In this case, the suspension is said to be $\textit{nonsparse}$. A suspension is $\textit{sparse}$ if it has no nonsparse leaf. 
	
	Let $\textup{Log}_1$ and $\textup{Log}_2$ be any complex logarithms on $\mathbb{C}$ with branch cuts $C_{1}:=\{z\in \mathbb{C}: \textup{Re}\,z\leq 0\}$ and $C_{2}:=\{z\in \mathbb{C}: \textup{Re}\,z\geq 0\}$, respectively. Let $(X,\lambda)$ be a vector field on $\mathbb{C}^n$ and $S^X_0(F)$ a suspension. Then for each $i\in \{1,2\}$, define
	\begin{align*}
		F'_{\lambda,i}:=&\bigg\{\bigg(\frac{z_2}{z^{\lambda_2}_1},\dots,\frac{z_n}{z^{\lambda_n}_1}\bigg)\in \mathbb{C}^{n-1}:(z_1,\dots,z_n)\in F,\,z_1\neq 0,\, z_1\notin C_{i}\bigg\},
	\end{align*}
	where $z^{\lambda_k}_1:=\textup{exp}\,(\lambda_k\cdot\textup{Log}_i\,z_1)$.  $S^{X}_0(F)$ is said to have a $\textit{regular leaf}~L_z$ generated by $z=(z_1,\ldots,z_n)\in \bar{F}$ if $z_1\neq 0$ and the $\lambda$-$\textit{direction set}$ $F'_{\lambda}:=F'_{\lambda,1}\cup F'_{\lambda,2}$ is locally $L$-regular at $\Big(\frac{z_2}{z^{\lambda_2}_1},\dots,\frac{z_n}{z^{\lambda_n}_1}\Big)\in F'_{\lambda,i}$ for some $i$. For the definition of $L$-regularity, see Definition \ref{Lregular}. A suspension is $\textit{regular}$ if it has a regular leaf.
\end{definition}
\subsection{Main theorem}
Let $\Psi_{E,\lambda}$ and $\rho_{\lambda}$ be the functions defined in Definition $\ref{definitionofextrftn}$. In this paper, we prove the following 
\begin{theorem}\label{main theorem}
	If a suspension $S^X_0(F)$ has a nonsparse leaf and a regular leaf, then it is a Forelli suspension; that is, any function $f:B^n\to \mathbb{C}$ satisfying the following two conditions
	\begin{enumerate}
		\setlength\itemsep{0.1em}
		\item $f\in C^{\infty}(0)$, and
		\item $t\in \mathbb{H}\to f\circ \Phi^{X}(z,t)$ is holomorphic for each $z\in F$
	\end{enumerate}
	is holomorphic on a $\lambda$-balanced domain of holomorphy
	\[
	\Omega:=\{z\in \mathbb{C}^n:\Psi^{\ast}_{S^X_0(F),\lambda}(z)<1\}\supset B^n(0;\{\rho_{\lambda}(S^{X}_0(F))\}^{\textup{max}(\lambda)})
	\]
	containing the origin. Furthermore, there exists an open neighborhood $U=U(F,X)$ $\subset S^{2n-1}$ of a generator $v_0\in \bar{F}$ of the regular leaf such that $f|_{\Omega}$ extends to a holomorphic function on an open set
	\[
	\hat{\Omega}=\{z\in \mathbb{C}^n:\Psi_{\Omega \cup S^X_0(U),\lambda}(z)<1\}
	\]
	which is the smallest $\lambda$-balanced domain of holomorphy containing $\Omega\cup S^X_0(U)$.
\end{theorem}

Here, the asterisk denotes the upper-semicontinuous regularization (\ref{uscregular}). As each point of a nonempty open subset $U$ of $S^{2n-1}$ generates a leaf that is both nonsparse and regular, $S^X_0(U)$ is always a Forelli suspension. This in particular improves the following 

\begin{theorem}[Kim-Poletsky-Schmalz \cite{KPS09}] \label{KPS}
	Let $X$ be a diagonalizable vector field of aligned type on $\mathbb{C}^n$. If $f:B^n\to \mathbb{C}$ satisfies the following two conditions
	\begin{enumerate}
		\item $f\in C^{\infty}(0)$, and
		\item $t\in \mathbb{H}\to f\circ \Phi^X(z,t)$ is holomorphic for each $z\in S^{2n-1}$,
	\end{enumerate}
	then $f$ is holomorphic on $B^n$.
\end{theorem}
If $X$ is the $\textit{complex Euler vector field}$, i.e., each eigenvalue of $X$ equals $1$, then Theorem \ref{KPS} reduces to the well-known analyticity theorem of Forelli.
\begin{theorem}[Forelli \cite{Forelli77}] \label{Forelli-original}
If $f:B^n\to \mathbb{C}$ satisfies the following two conditions
\begin{enumerate}
	\item $f\in C^{\infty}(0)$, and
	\item $t\in B^1 \to f(tz)$ is holomorphic for each $z\in S^{2n-1}$,
\end{enumerate}
then $f$ is holomorphic on $B^n$.
\end{theorem}
 At this point, a few features of Theorem \ref{main theorem} should be worth mentioning. First, the analyticity of the given function $f$ depends on the local behavior of $f$ near the two specific leaves of $S^X_0(F)$. Therefore, the theorem can be regarded as a localization of Theorem \ref{KPS}. Second, the suspension in Theorem \ref{main theorem} needs not to be generated by an open subset of $S^{2n-1}$ in general; we will construct a nowhere dense Forelli suspension in Example \ref{nowhere dense Forelli suspension}. Note also that $\hat{\Omega}$ depends only on $F$ and $X$. Finally, the examples in Section \ref{Sect example} indicate that a formal Forelli suspension is not necessarily normal, nor vice versa. 

 \subsection{Structure of paper, and remarks}
 
 The original version of Forelli's theorem in \cite{Forelli77} is concerned with functions harmonic along the set of complex lines passing through the origin. But as noted in \cite{Stoll80}, the proof arguments in \cite{Forelli77} also imply Theorem \ref{Forelli-original}; if $f:B^n\to \mathbb{C}$ is the given function, then one may proceed in two steps as follows:
 
 \smallskip
 \begin{narrower}
 	
 	\textbf{Step 1.} The formal Taylor series $S_f$ of $f$ is of holomorphic type.
 	
 	\textbf{Step 2.} The formal series $S_f$ converges uniformly on some $B^n(0;r)$.
 	
 \end{narrower}
 \smallskip
\noindent
 Then by Hartogs' lemma (Lemma \ref{Hartogslemma}),  $f\equiv S_f$ is holomorphic on $B^n$.
 
 Several works related to Step 2 originate from a question of Bochner which was answered affirmatively by Zorn \cite{Zorn47}, Ree \cite{Ree49}, Lelong \cite{Lelong51}, and Cho-Kim \cite{CK21}. On the other hand, Siciak provided a complete solution to Leja's problem on the uniform convergence of a formal sum of homogeneous polynomials in \cite{Siciak90}. In the solution, the theory of projective capacity and related extremal function developed in \cite{Siciak82} played crucial roles (cf. \cite{LevenMol88}, \cite{Sadullaev22}).
 
 There have been many attempts to weaken the condition $f\in C^{\infty}(0)$ in Theorem \ref{Forelli-original} to finite differentiability. Although no success was possible (see \cite{JKS16} for counterexamples),  Condition (2) has been generalized successfully to various directions, starting with \cite{Chirka06}. See also \cite{JKS13}, \cite{CK21}. In particular, it was shown in \cite{KPS09} that the set of integral curves of a diagonalizable vector field can replace the set of complex lines in Theorem \ref{Forelli-original} if, and only if, the field is aligned. This was generalized to the case of nondiagonalizable vector fields contracting at the origin in \cite{JKS16}.  Then at this juncture, it would be natural to address the following
 \begin{problem*}
 	Let $\mathcal{F}$ be the set of integral curves of a contracting vector field $X$ of aligned type. Characterize the local properties of a set $\mathcal{F}'\subset \mathcal{F}$ for which the following holds: a function $f:B^n\to \mathbb{C}$ that is $\textup{(}1\textup{)}$ smooth at the origin, and $\textup{(}2\textup{)}$ holomorphic along each curve in $\mathcal{F}'$ is holomorphic on a neighborhood of the origin.
 \end{problem*}

   When $X$ is the complex Euler vector field, the author followed the original steps of Forelli and provided an answer to the problem in \cite{Cho22}. In this paper, we extend the proofs in \cite{Cho22} to the case where $X$ is a general diagonalizable vector field. Once Step 1 is achieved, the proof of Theorem \ref{main theorem} reduces to showing the uniform convergence of a formal sum of quasi-homogeneous polynomials. So following Siciak \cite{Siciak82}, we develop a new capacity theory and use it with the methods in \cite{KPS09},  \cite{CK21} to establish Step 2. Then the conclusion follows from Hartogs' lemma in \cite{Shiffman89} and (\ref{domofholoext}). Although Step 2 can also be settled without the capacity theory as Theorem \ref{Sadullaev} shows, our proofs in particular provide 
  \begin{enumerate}
  	\setlength\itemsep{0.1em}
  	\item a complete characterization of normal suspensions generated by $F_{\sigma}$ sets in $S^{2n-1}$ (Theorem \ref{normalsuspension}), 
  	\item an explicit description of the polynomially convex hull of a $\lambda$-circular set (Theorem \ref{polyhullmain}, see Definition \ref{lambdacircular} for the definition of $\lambda$-circular set), and
  	\item analytic continuation of a holomorphic function on an open set to the smallest $\lambda$-balanced domain of holomorphy containing the open set (Proposition \ref{domainofholo1}, Theorem \ref{smallestdomainofholo}).
  \end{enumerate}

 On the other hand, we remark that we do not know how to carry out the arguments when the given suspension is generated by a nondiagonalizable contracting vector field of aligned type.
\subsection*{Acknowledgement}
The author would like to thank Professor Taeyong Ahn and Dr. Seungjae Lee for their helpful comments. Most parts of the paper were written while the author was supported by the National Research Foundation of Korea (NRF-2018R1C1B3005963,  NRF-2021R1A4A1032418). The author is currently supported by the National Research Foundation of Korea (NRF-2021R1A4A1032418, NRF-2023R1A2C1007227). 
\section*{Statements and Declarations}
\subsection*{Conflict of interests} The author states that there is no conflict of interests.
\subsection*{Data availability statement} This article has not used any associated data.
\section{Asymptotic expansions}\label{Sect asymptotic}
In this section,  we summarize the properties of asymptotic expansions introduced in \cite{KPS09}. In particular, Proposition \ref{asympholo} will replace the classical Cauchy estimate throughout the paper.
\begin{definition}[\cite{KPS09}]
	\normalfont
 Suppose that $\{\rho_j\}$, $j\geq 0$, is a strictly increasing sequence of nonnegative real numbers converging to infinity with $\rho_0=0$ and let $\{n_j\}$, $j\geq 0$,  be a sequence of nonnegative integers. A formal series
	\begin{equation}
		\sum_{j=0}^{\infty}\sum_{k=0}^{n_j}p_{jk}e^{-\mu_{jk}t-\nu_{jk}\bar{t}}
	\end{equation}
	is called an $\textit{asymptotic expansion}$ of a function $f:\mathbb{H}\to \mathbb{C}$ if (1) $\mu_{jk},\,\nu_{jk}\geq 0$, $\mu_{jk}+\nu_{jk}=\rho_j$ for every $j$ and $k$ and (2) for every $n$, we have
	\[
	\bigg|f(t)-\sum_{j=0}^{n}\sum_{k=0}^{n_j}p_{jk}e^{-\mu_{jk}t-\nu_{jk}\bar{t}}\bigg|\,e^{\rho_n\textup{Re}\,t}\to 0
	\]
	as $\textup{Re}\,t\to \infty$ in $\mathbb{H}$.
\end{definition}

It is known that every function $f:\mathbb{H}\to \mathbb{C}$ has at most one asymptotic expansion; see Proposition 2.3 in \cite{KPS09}. Let $z=(z_1,\dots,z_n)=(x_1,y_1,\dots,x_n,y_n)$ be the standard complex coordinate system on $\mathbb{C}^n,$ where $z_j=x_j+iy_j$ for each $j\in \{1,\dots,n\}$. Recall the multi-index notation as follows: 
\begin{align*}
	k=(k_1,\ldots,k_n),~|k|=k_1+\cdots+k_n,~k!=k_1!\cdots k_n!,\,\text{and}\,z^{k}=z_1^{k_1} \cdots z_n^{k_n}.
\end{align*}
We say that $f:B^n\to \mathbb{C}$ has a $\textit{formal Taylor series}$ $S$ at the origin if
\begin{equation}\label{formalseries}
	S=\sum_{j=0}^{\infty}\sum_{|k|+|m|=j}a_{km}z^k\bar{z}^m
\end{equation}
is a formal series such that for every $n$, we have
\[
\Bigg|f(z)-\sum_{j=0}^{n}\sum_{|k|+|m|=j}a_{km}z^k\bar{z}^m \Bigg|=o(\|z\|^n).
\]
If $f\in C^{\infty}(0)$, then $f$ has a formal Taylor series whose coefficients are given as
\begin{equation}\label{formal Taylor series}
	a_{km}:=\frac{1}{k!\,m!}\frac{\partial^{|k|+|m|}f}{\partial z^k\partial \bar{z}^{m}}(0).
\end{equation}
\begin{proposition}[\cite{KPS09}]\label{asympTaylor}
	 If $f:B^n\to \mathbb{C}$ has a formal Taylor series $\textup{(}\ref{formalseries}\textup{)}$ at the origin, then the function $f_z:t\in \mathbb{H}\to f\circ \Phi^{X}(z,t)$ has the asymptotic expansion
	\[
	\sum_{j=0}^{\infty}\Bigg(\sum_{(\lambda,k)+(\lambda,m)=\rho_j}a_{km}z^k\bar{z}^me^{-(\lambda,k)t-(\lambda,m)\bar{t}} \Bigg)
	\]
	on $\mathbb{H}$ for each $z\in \mathbb{C}^n$, where $\{\rho_j\}$ is the increasing sequence of all possible values
	\begin{equation}\label{defofrho_j}
		(\lambda,k)+(\lambda,m)=\lambda_1k_1+\cdots+\lambda_n k_n+\lambda_1 m_1+\cdots +\lambda_nm_n.
	\end{equation}
	Furthermore, if $f_z$ is holomorphic for some $z\in \mathbb{C}^n$, then the asymptotic expansion of $f_z$ does not contain nonholomorphic terms.
\end{proposition} 
\begin{proposition}[\cite{KPS09}]\label{asympholo}
	Let $f:\mathbb{H}\to \mathbb{C}$ be a holomorphic function with an asymptotic expansion $\sum_{j=0}^{\infty}c_je^{-\mu_jt}$. If $|f|\leq M$, then $|c_j|\leq M$ for each $j$.
\end{proposition}

We will use the following lemma in Section \ref{Sect properties of ext ftns}.
\begin{lemma}\label{convergence}
	Let  $(X,\lambda)$ be a a vector field on $\mathbb{C}^n$ and $\{\rho_j\}$ the increasing sequence of all possible values in \textup{(}\ref{defofrho_j}\textup{)} with $m=0$. If $\{a_j\}\subset \mathbb{C}$ is a sequence such that 
	\[
	r:=\limsup\limits_{j\to \infty}|a_j|^{\frac{1}{\rho_j}}<1,
	\]
	 then the series $S:=\sum_{j=1}^{\infty}a_j$ converges to a finite complex number. If $r>1$, then the series diverges.
\end{lemma}
\begin{proof}
	We first suppose that $r<1$ and prove that $S$ is convergent. Set $s:=\frac{1+r}{2}<1$. By the assumption, there is an integer $N>0$ such that $|a_j|<s^{\rho_j}$ whenever $j\geq N$. So it suffices to show that $\sum_{j=1}^{\infty}s^{\rho_j}$ converges. As $\lambda_1=1$, one can choose an increasing sequence $\{\ell_j\}$ of positive integers such that $\rho_{\ell_j}=j$.  For each positive integer $j$ and $k\in \{1,\dots,n\}$, choose a nonnegative integer $m^{(j)}_k$ such that 
	\begin{equation}\label{inequality}
		\lambda_{k}m_k^{(j)}\leq  j < \lambda_{k}m^{(j)}_k+\lambda_k.
	\end{equation}
Then letting $j=1$ in (\ref{inequality}) and multiplying each side of the inequality by $j$, we obtain 
	$j\lambda_{k}m^{(1)}_k\leq j < j\lambda_{k}m^{(1)}_k+j\lambda_k$. So it follows from the preceding inequalities that 
	\[
	m^{(j)}_k \leq jm^{(1)}_k+j.
	\] 
	Note also that, by (\ref{inequality}), we have  $\lambda_1m_1+\cdots+ \lambda_n m_n\leq j$ only if $m_k\leq m^{(j)}_k$ for each $k\in \{1,\dots,n\}$. Then
	\begin{equation}\label{indexgrowth}
		{\ell_j}\leq \prod\limits_{k=1}^{n} (m^{(j)}_{k}+1)\leq \prod\limits_{k=1}^{n} (jm^{(1)}_k+j+1):=p_n(j),
	\end{equation}
	where $p_n$ is a real polynomial of degree at most $n$. Therefore,  
	\begin{align*}
		\sum\limits_{j=1}^{\infty}s^{\rho_j}&=\sum\limits_{j=1}^{\ell_1}s^{\rho_j}+\sum\limits_{j=\ell_1}^{\infty}s^{\rho_j}\leq \sum\limits_{j=1}^{\ell_1}s^{\rho_j}+\sum\limits_{j=1}^{\infty}(\ell_{j+1}-\ell_j)s^{\rho_{\ell_j}}\\
		&\leq \sum\limits_{j=1}^{\ell_1}s^{\rho_j}+\sum\limits_{j=1}^{\infty}p_n(j+1)\cdot s^j<\infty 
	\end{align*}
	as desired. 
	
	If $r>1$, then one can find a subsequence $\{a_{n_j}\}$ of $\{a_j\}$ such that $|a_{n_j}|>1$ for each $j$. So the series $S$ diverges. 
\end{proof}

\section{Formal Forelli suspensions}
We first settle the following characterization of formal Forelli suspensions.
\begin{theorem}\label{formal Forelli suspension theorem}
	 A suspension is a formal Forelli suspension if, and only if, it has a nonsparse leaf.
\end{theorem}
\begin{proof}
	
First, we prove that $S^X_0(F)$ is not a formal Forelli suspension under the assumption that $S^X_0(F)$ is sparse. Then for each $v\in \bar{F}$, there exist an open neighborhood $U_v\subset S^{2n-1}$ of $v$ and a polynomial $q_v\in \mathcal{H}_{\lambda}$ with $\textup{bideg}\,q_v=(d_1,d_2)$, $d_2\neq 0$, such that $q_v\equiv 0$ on $\bar{F}\cap U_v$. Since $\mathcal{U}:=\{U_v:v\in \bar{F}\}$ is an open cover of the compact set $\bar{F}$, there is a finite subcover $\{U_{v_1},\dots,U_{v_m}\}$ of $\mathcal{U}$. Then the polynomial $q:=q_{v_1}\cdot q_{v_2}\cdots q_{v_m}\in \mathcal{H}_{\lambda}$ has a fixed bidegree $(d'_1,d'_2)$, $d'_2\neq 0$. Note that $q$ is smooth and holomorphic along $S^X_0(F)$ but it is not of holomorphic type.
	
	Conversely, suppose that $S^X_0(F)$ has a nonsparse leaf and let $f:B^n\to \mathbb{C}$ be a function satisfying the following two conditions:
	\begin{enumerate}
		\item $f\in C^{\infty}(0)$, and
		\item $f$ is holomorphic along $S^X_0(F)$.
	\end{enumerate}
	Then we are to show that $f$ has a formal Taylor series 
	$S$ of holomorphic type. Recall that each coefficient $a_{km}$ of $S$ is given by (\ref{formal Taylor series}). By Proposition \ref{asympTaylor}, we have
	\begin{equation}\label{linear system}
		S_{\nu}^{\mu}(z):=\sum_{\substack{(\lambda,k)=\mu \\ (\lambda,m)=\nu }} a_{km}z^k\bar{z}^{m}=0~\forall z\in F
	\end{equation}
	for each fixed $\mu$ and $\nu \neq 0$. As the suspension $S^X_0(F)$ has a nonsparse leaf and the polynomial $S_{\nu}^{\mu}\in \mathcal{H}_{\lambda}$ with bidegree $(\mu, \nu)$ vanishes on $F$, we have $S_{\nu}^{\mu}\equiv 0$ on $\mathbb{C}^n$. Therefore, $a_{km}=0$ whenever $m\neq 0$ as desired.
\end{proof}

\begin{corollary}\label{opensuspension}
	If $U\subset S^{2n-1}$ is a nonempty open set, then $S^X_0(U)$ is a formal Forelli suspension for any $X$.
\end{corollary}
\begin{proof}
	We show that any point $z\in U$ generates a nonsparse leaf of $S^X_0(U)$. Let $q\in \mathcal{H}_{\lambda}$ be a polynomial with $\textup{bideg}\,q=(d_1,d_2),$ $d_2\neq 0$. If $q\equiv 0$ on $U\cap V$ for a nonempty open neighborhood $V\subset S^{2n-1}$ of $z$, then $q \equiv 0$ on $S^X_0(U\cap V)$. Since $S^X_0(U\cap V)$ is open by the rectification theorem \cite{IY07} and $\textup{Re}\,q,\,\textup{Im}\,q$ are real-analytic, we have $q\equiv 0$ on $\mathbb{C}^n$. 
\end{proof}

 For convenience, we say that $\lambda=(\lambda_1,\dots,\lambda_n)$ is linearly (in)dependent over the ring $\mathbb{Z}$ of integers if the set $\{\lambda_1,\dots,\lambda_n\}$ is so.
\begin{proposition}\label{transcendentalsparseness}
	 Let $F\subset S^{2n-1}$ be a nonempty set and $(X,\lambda)$ a vector field on $\mathbb{C}^n$. If $\lambda$ is linearly independent over $\mathbb{Z}$, then any point $w=(w_1,\dots,w_n)$ $\in \bar{F}$ satisfying $w_k\neq 0$ for each $k$ generates a nonsparse leaf of $S^{X}_0(F)$. Conversely, if $\lambda$ is linearly dependent over $\mathbb{Z}$, then there exists a sparse suspension $S^X_0(G)$ containing $\Big(\frac{e^{-\lambda_1}}{\sqrt{n}},\dots,\frac{e^{-\lambda_n}}{\sqrt{n}}\Big)$.
\end{proposition}
\begin{proof}
	Let $U$ be an open neighborhood of $w$ in $S^{2n-1}$ and choose $q\in \mathcal{H}_{\lambda}$ with $\textup{bideg}\,q=(d_1,d_2)$, $d_2\neq0$ such that $q\equiv 0$ on $\bar{F}\cap U$. Then $q$ is a finite sum of monomials taken over all multi-indices $k,m$ satisfying 
	\begin{equation*}
		\begin{cases}
			\lambda_1 k_1+\cdots+\lambda_n k_n =d_1 \\
			\lambda_1 m_1+\cdots+\lambda_n m_n =d_2.
		\end{cases}
	\end{equation*}
 So it follows from the linear independence of $\lambda$ that the equation has a unique solution if any exists. Therefore, $q$ is a monomial and the condition $q(w)=0$ implies that $q\equiv 0$ on $\mathbb{C}^n$ as desired.
	
	Suppose that $\lambda$ is linearly dependent over $\mathbb{Z}$. Then one can assume that there exist nonnegative integers $1\leq r<s\leq n,$ $\alpha_1,\dots,\alpha_r,$ $\beta_{r+1},\dots,\beta_{s}$ such that $\alpha_k\neq 0,\,\beta_{\ell}\neq 0$ for some $k,\ell$, and
	\[
	\alpha_1 \lambda_1+\cdots+\alpha_r \lambda_r = \beta_{r+1}\lambda_{r+1}+\cdots+\beta_s \lambda_s:=\gamma>0.
	\]
	Define $G:=\{(z_1,\dots,z_n)\in S^{2n-1}:\textup{Im}\,z_i=0\;\text{for each}\;i\in \{1,\dots,n\}\}$ and
	\[
	q(z):=\textup{Im}\,(z^{\alpha_1}_1\cdots z^{\alpha_r}_r \cdot \bar{z}^{\beta_{r+1}}_{r+1}\cdots \bar{z}_s^{\beta_s})\in \mathbb{C}[z_1,\dots,z_n,\bar{z}_1,\dots,\bar{z}_n]
	\]
	so that $q\in \mathcal{H}_{\lambda}$ and $\textup{bideg}\,q=(\gamma,\gamma).$ Since $q\equiv 0$ on $G$, $S^X_0(G)$ is sparse. Choose
	\[
	z_0:=\Big(\frac{1}{\sqrt{n}},\dots,\frac{1}{\sqrt{n}}\Big)\in G.
	\]
	 Then $\Phi^X(z_0,1)=\Big(\frac{e^{-\lambda_1}}{\sqrt{n}},\dots,\frac{e^{-\lambda_n}}{\sqrt{n}}\Big)\in S^X_0(G)$ and this completes the proof.
\end{proof}
\begin{corollary}
	Let $X$ be a vector field on $\mathbb{C}^n$ with eigenvalues $(1,\lambda,\lambda^2,\dots,\lambda^{n-1})$, where $\lambda>0$  is a transcendental number. If a point $w=(w_1,\dots,w_n)\in \bar{F}\subset S^{2n-1}$ satisfies $w_k\neq 0$ for each $k$, then $w$ generates a nonsparse leaf of $S^{X}_0(F)$. The statement also holds if $n=2$ and $\lambda$ is a positive irrational number.
\end{corollary}

 The following example, together with Proposition \ref{transcendentalsparseness}, illustrates that the sparseness of $S^X_0(F)$ depends on both $F$ and $X$.
\begin{example}
	\normalfont 
	Identify $\mathbb{R}^4$ with $\{(z_1,z_2,z_3)\in \mathbb{C}^3: \textup{Im}\,z_1=\textup{Im}\,z_2=0\}$ and define
	\begin{gather*}
		F:=\{(x,y,z)\in \mathbb{R}^4\cap S^{5}:\,x,y\in \mathbb{R},\,z\in \mathbb{C}\}.
	\end{gather*}
    Let $(X,\lambda)$ be a vector field on $\mathbb{C}^3$ with eigenvalues $\lambda=(1,\lambda_2,\lambda_3).$ If $\lambda_2$ is a positive integer, then $S^{X}_0(F)$ is sparse as $F\subset Z(z^{\lambda_2}_1\bar{z}_2-\bar{z}^{\lambda_2}_1z_2)$. 
	
	Now suppose that $\lambda_2$ is irrational. Then we show that $v:=(1,0,0)\in F$ generates a nonsparse leaf of $S^X_0(F)$. Choose an open neighborhood $U$ of $v=(1,0,0)\in F$ in $S^5$ and suppose that $S^X_0(\bar{F}\cap U)\subset Z(q)$ for some $q\in \mathcal{H}_{\lambda}$ with $\textup{bideg}\,q=(d_1,d_2)$, $d_2\neq 0$. Note that $q$ can be written as	
	\begin{equation}\label{identity}
		q=\sum_{{\substack{(\lambda,k)=d_1 \\ \substack{(\lambda,m)=d_2}}} }a_{km}\, w^k\bar{w}^m,\, w\in \mathbb{C}^3
	\end{equation}
	where $\{a_{km}\}$ is a finite set of complex numbers. Then we are to show that $q\equiv 0$ on $\mathbb{C}^3$. By the given assumption, we have $q(x,y,z)=0$ for each $(x,y,z)\in F\cap U$ and this translates into the following equation:
     \begin{equation}\label{vanishing2}
     	\sum_{{\substack{(\lambda,k)=d_1 \\ \substack{(\lambda,m)=d_2}}} }a_{km}\,x^{k_1+m_1}y^{k_2+m_2}z^{k_3}\bar{z}^{m_3}=0
     \end{equation}
     for each $(x,y,z)\in F\cap U.$ Choose nonnegative integers $r_1,r_2,s_1,s_2$. By the identity theorem for polynomials, (\ref{vanishing2}) reduces to the equation $\sum a_{km}=0$, where the sum is taken over all multi-indices $k,m$ satisfying
	\begin{equation}\label{vanishing3}
		\begin{cases}
			k_1+\lambda_2 k_2+\lambda_3 k_3 = d_1 \\
			m_1+\lambda_2 m_2+\lambda_3 m_3 = d_2 \\
			k_1+m_1 = r_1\\
			k_2+m_2 = r_2\\
			k_3 = s_1 \\
			m_3 = s_2.
		\end{cases}
	\end{equation}
Since $(1,\lambda_2)$ is linearly independent over $\mathbb{Z}$, (\ref{vanishing3}) has a unique solution if any exists. So the equation $\sum a_{km}=0$ implies that $a_{km}=0$ for each $k,m$ appearing in (\ref{identity}). Therefore, $q\equiv 0$ on $\mathbb{C}^3$ and $S^X_0(F)$ is nonsparse as desired.
\end{example}

\section{Pluripotential theory}\label{Sect properties of ext ftns}

\begin{definition}\label{definitionofextrftn}
	\normalfont
	Let $(X,\lambda)$ be a vector field on $\mathbb{C}^n$. Define a set  $H_{\lambda}$ of nonconstant functions as
	\[
	H_{\lambda}:=\{u\in \textup{PSH}(\mathbb{C}^n): u\geq 0~\text{on}~\mathbb{C}^n, \,u(\Phi^{X}(z,t))=e^{-\textup{Re}\,t}\cdot u(z)\; \forall z\in \mathbb{C}^n, t\in \mathbb{C}\}.
	\]
	For each bounded subset $E$ of $\mathbb{C}^n$, define
	\[
	\Psi_{E,\lambda}(z):=\textup{sup}\,\{u(z):u\in H_{\lambda},\; u\leq 1~\text{on}~E\}~\text{for each}~z\in \mathbb{C}^n.
	\]
	If $E$ is unbounded, then we set
	\[
	\Psi_{E,\lambda}(z):=\textup{inf}\,\{\Psi_{F,\lambda}(z): F\subset E~\text{is bounded} \}~\text{for each}~z\in \mathbb{C}^n.
	\]	
	The $\lambda\textit{-projective capacity}$ of a set $E\subset \mathbb{C}^n$ is defined as
	\[
	\rho_{\lambda}(E):=\textup{inf}\,\{\|u\|_{E}:u\in H_{\lambda},\, \|u\|_{S^{2n-1}}=1\},~\text{where}~\|u\|_{E}:=\sup\limits_{z\in E}|u(z)|.
	\]
\end{definition}
If $(X,\lambda)$ is the complex Euler vector field, then $\Psi_{E,\lambda}$ and $\rho_{\lambda}$ reduce to the extremal function and the projective capacity introduced in \cite{Siciak82}, respectively. 

This section is organized as follows. In Subsection \ref{subsection 4.1}, we formulate methods (Theorem \ref{quasihomstar}, Theorem \ref{mainapproximation}) for approximating a function in $H_{\lambda}$ by quasi-homogeneous polynomials of type $\lambda$. In Subsection \ref{subsection4.2}, we study the basic properties of $\lambda$-pluripolar sets. Then the results in the two subsections will be used to develop the theory of the $\lambda$-projective capacity and the related extremal function in Subsection \ref{subsection 4.3}. The whole theory culminates in the characterization of $\lambda$-pluripolar sets in terms of $\rho_{\lambda}, \Psi_{E,\lambda}$ (Theorem \ref{plurigreenandextgeneral}, Theorem \ref{characterizationofHppsets}).

Most of the arguments in this section follow the methods of \cite{Siciak82}. But we try to give the proofs in detail as \cite{Siciak82} seems not to be easily accessible.
\subsection{Plurisubharmonic functions on $\mathbb{C}^n$ generated by quasi-homogeneous polynomials}\label{subsection 4.1}
The following lemma of Hartogs will be important throughout. 
\begin{lemma}[Hartogs \cite{Hartogs1906}]\label{Hartogslemma}
	Let $\{u_m\}$ be a sequence of subharmonic functions on an open set $\Omega\subset \mathbb{C}^n$ and $C\in \mathbb{R}$ a constant such that
	\begin{enumerate}
	\setlength\itemsep{0.1em}
	\item $\{u_m\}$ is locally uniformly bounded from above on $\Omega$, and
	\item $\limsup\limits_{m\to \infty}u_m(z)\leq C$ for any $z\in \Omega$.
	\end{enumerate}
	 If $K$ is a compact subset of $\Omega$ and $\epsilon$ is a positive number, then there exists a positive integer $N=N(K,\epsilon)$ such that $u_m(z)\leq C+\epsilon$ whenever $m\geq N$ and $z\in K$.
\end{lemma}

For the proof of the lemma, see \cite{Nara95}. Given a vector field $(X,\lambda)$ on $\mathbb{C}^n$, we denote by $\{\rho_j\}$ the increasing sequence of all possible values in (\ref{defofrho_j}) with $m=0$.
\begin{theorem}\label{quasihomexp}
Let $\Omega\subset \mathbb{C}^n$ be a $\lambda$-balanced domain containing the origin. If $f:\Omega\to \mathbb{C}$ is holomorphic, then there exists a sequence $\{q_m\}\subset \mathcal{H}_{\lambda}$ with $\textup{bideg}\,q_m=(\rho_m,0)$ such that $f=\sum_{m=0}^{\infty}q_m$ on $\Omega$. 
\end{theorem}

\begin{proof}
Let $B_r:=B^n(0;r)\subset \Omega$ be an open ball such that $\bar{B}_r\subset \Omega$. Then one can choose a sequence $\{q_m\}\subset \mathcal{H}_{\lambda}$ with $\textup{bideg}\,q_m=(\rho_m,0)$ satisfying   $ f=\sum_{m=0}^{\infty}q_m~\text{on}~\bar{B}_r.$  Since $\Omega$ is $\lambda$-balanced, the map $t\in \bar{\mathbb{H}}\to f(\Phi^X(z,t))$ is a well-defined bounded map for any $z\in \bar{B}_r$. So by Proposition \ref{asympholo}, we have $\|q_m\|_{\bar{B}_r}\leq \|f\|_{\bar{B}_r}$ for each $m\geq 0$. Note that 
\begin{equation}\label{degreeandrho_k}
\textup{deg}\,q_m\cdot \textup{min}(\lambda)\leq \rho_m \leq \textup{deg}\,q_m\cdot \textup{max}(\lambda),
\end{equation}
where $\textup{max}(\lambda):=\textup{max}\{\lambda_1,\dots,\lambda_n\}$ and $\textup{min}(\lambda):=\textup{min}\{\lambda_1,\dots,\lambda_n\}$. Then recall that the following $\textit{Bernstein-Walsh inequality}$ holds for any $q\in \mathbb{C}[z_1,\dots,z_n]$:
\begin{equation}{\label{classical BW inequality}}
  |q(z)|\leq \|q\|_{\bar{B}_r}\cdot \bigg\{\textup{max}\,\bigg(1,\frac{\|z\|}{r}\bigg)\bigg\}^{\textup{deg}\,q}~\text{for each}~z\in \mathbb{C}^n.
\end{equation}
So we have
\begin{align*}
u_m(z):=|q_m(z)|^{\frac{1}{\rho_m}} &\leq \|f\|^{\frac{1}{\rho_m}}_{\bar{B}_r}\cdot \bigg\{\textup{max}\,\bigg(1,\frac{\|z\|}{r}\bigg)\bigg\}^{\frac{\textup{deg}\,q_m}{\rho_m}}\\
&\leq \|f\|^{\frac{1}{\rho_m}}_{\bar{B}_r}\cdot\bigg\{\textup{max}\,\bigg(1,\frac{\|z\|}{r}\bigg)\bigg\}^{\frac{1}{\textup{min}(\lambda)}} 
\end{align*}
for each $m\geq 0$ and $z\in \mathbb{C}^n$. Therefore, the sequence $\{u_m\}\subset \textup{PSH}(\mathbb{C}^n)$ is locally uniformly bounded from above. Choose $z\in \Omega$ and $b=b(z)>0$ such that $\Phi^X(z,t)\in B$ for each $t\in \mathbb{H}$ with $\textup{Re}\,t>b$. Then
\begin{equation}\label{asymseries}
	f(\Phi^X(z,t))=\sum_{m=0}^{\infty}q_m(z)e^{-\rho_m t}
\end{equation}
if  $t\in \mathbb{H}$ and $\textup{Re}\,t>b$. Note that the series in (\ref{asymseries}) converges for any $t\in \mathbb{H}$. So it follows from Lemma \ref{convergence} that
\begin{equation*}
\limsup\limits_{m\to \infty}u_m(z)\cdot e^{-\textup{Re}\,t}\leq 1~\text{if}~z\in \Omega, t\in \mathbb{H}.
\end{equation*}
Letting $t \to 0$, we obtain $\limsup_{m\to \infty}u_m(z)\leq 1$ for each $z\in \Omega$. Let $K\subset \Omega$ be a compact set. As $\Omega$ is $\lambda$-balanced, there exists a number $t_0>0$ such that the set 
\[
K_{t_0}:=\{\Phi^X(z,-2t_0):z\in K\}\subset \Omega
\]
is relatively compact in $\Omega$. Then by Lemma \ref{Hartogslemma}, there exists a number $N_0>0$ such that $u_m(z)\leq e^{t_0}$ whenever $m\geq N_0,\, z\in K_{t_0}.$
So if  $m\geq N_0$ and $z\in K$, then
\begin{align*}
|q_m(z)|=e^{-2\rho_m t_0}\cdot |q_m(\Phi^X(z,-2t_0))|\leq e^{-\rho_m t_0}.
\end{align*}
Therefore, $S:=\sum_{m=0}^{\infty}q_m$ converges uniformly on $K$ by Lemma \ref{convergence} and it defines a holomorphic function on $\Omega$. Since $f\equiv S$ on $B$, it follows from the principle of analytic continuation that $f\equiv S$ on $\Omega$.
\end{proof}
Let $u:\Omega\to [-\infty,\infty)$ be a function defined on an open set $\Omega\subset \mathbb{C}^n$. The $\textit{upper-semicontinuous regularization}~u^{\ast}:\Omega\to [-\infty,\infty)$  of $u$ is defined to be
\begin{equation}\label{uscregular}
u^{\ast}(z):=\limsup_{\Omega \ni w\to z}u(w)~\forall z\in \Omega.
\end{equation}
\begin{theorem}\label{quasihomstar}
Let $u:\mathbb{C}^n\to [0,\infty)$ be a given function. Then $u\in H_{\lambda}$ if, and only if, there exists a sequence $\{q_m\}\subset \mathcal{H}_{\lambda}$ with $\textup{bideg}\,q_m=(\rho_m,0)$ such that 
\begin{equation}\label{quasihom}
u=\Big(\limsup\limits_{m\to \infty}|q_m|^{\frac{1}{\rho_m}}\Big)^{\ast}~\text{on}~\mathbb{C}^n.
\end{equation}
\end{theorem} 
\begin{proof}
Suppose that (\ref{quasihom}) holds and define a set $A_m:=\{z\in \mathbb{C}^n: u(z)\leq m\}$ for each $m\geq 1$. Then note that $\mathbb{C}^n=\bigcup_{m=1}^{\infty}A_m.$ By the Baire category theorem, $A_M$ has a nonempty interior for some $M\geq 1$. So there exist an open ball $B^n(a;r)$ and $N>0$ such that  $u_m(z):=|q_m(z)|^{\frac{1}{\rho_m}}\leq N$ for each $m\geq 0,~z\in B^n(a;r).$ Then it follows from (\ref{classical BW inequality}) that the sequence $\{u_m\}$ is locally uniformly bounded. Applying Fatou's lemma to the submean inequality for $u_m$, we conclude that $u$ is plurisubharmonic and $u\in H_{\lambda}$.

Conversely, suppose that $u\in H_{\lambda}$ and define a $\lambda$-balanced domain $\Omega:=\{z\in \mathbb{C}^n:u(z)<1\}$ containing the origin. Then it is well-known that $\Omega$ is a domain of holomorphy, i.e., there exists a holomorphic function $f:\Omega\to \mathbb{C}$ that cannot be extended holomorphically across the boundary $\partial \Omega$ of $\Omega$. By Theorem $\ref{quasihomexp}$, there is a sequence $\{q_m\}\subset \mathcal{H}_{\lambda}$ with $\textup{bideg}\,q_m=(\rho_m,0)$ such that $f=\sum_{m=0}^{\infty}q_m~\text{on}~\Omega.$ Define a function $v\in H_{\lambda}$ as
\[
v(z):=\Big(\limsup_{m\to \infty}|q_m(z)|^{\frac{1}{\rho_m}}\Big)^{\ast}.
\]
Then by Lemma \ref{convergence} and the choice of $f$, we have $v(z)<1$ if, and only if, $z\in \Omega$. So $\Omega=\{z\in \mathbb{C}^n:v(z)<1\}$. Note also that $u(z_0)=v(z_0)=1$ whenever $z_0\in \partial \Omega$. Therefore, $u\equiv v$ on $\{\Phi^{X}(z_0,t):t\in \mathbb{C}\}$. Since the set of integral curves of $X$ forms a foliation of $\mathbb{C}^n-\{0\}$, we have $u\equiv v$ on $\mathbb{C}^n$ as desired.
\end{proof}
\begin{remark}\label{remarklogofHlambdaftn}
\normalfont
Equation (\ref{quasihom}) implies that $\textup{log}\,u\in \textup{PSH}(\mathbb{C}^n)$ for any $u\in H_{\lambda}$. So if $u,v\in H_{\lambda}$, then
\[	u^{\alpha}v^{\beta}=\textup{exp}\,(\alpha\,\textup{log}\,u+\beta\,\textup{log}\,v)\in\textup{PSH}(\mathbb{C}^n)
\]
whenever $\alpha,\beta$ are nonnegative numbers. Note also that $u^{\alpha}v^{\beta}\in H_{\lambda}$ if $\alpha+\beta=1$.
\end{remark}

It turns out that the approximation (\ref{quasihom}) is of limited use as the equation involves the upper-semicontinuous regularization. To develop a better approximation theorem for functions in $H_{\lambda}$, we first introduce the following

\begin{definition}\label{lambdacircular}
\normalfont
Let $(X,\lambda)$ be a vector field on $\mathbb{C}^n$ and $K$ a compact subset of a $\lambda$-balanced domain $\Omega\subset\mathbb{C}^n$. The $\textit{polynomially}$ $\lambda$-$\textit{convex hull}$ of $K$ in $\Omega$ is
\[
\hat{K}_\lambda:=\{z\in \Omega: |q_m(z)|\leq \|q_m\|_{K}~\text{for any}~q_m\in \mathcal{H}_{\lambda},\,\textup{bideg}\,q_m=(\rho_m,0)\}.
\]
We say that a set $E\subset \mathbb{C}^n$ is $\lambda$-$\textit{circular}$ if $\Phi^X(z,t)\in E$ whenever $z\in E$ and $t\in \mathbb{C},~\textup{Re}\,t=0.$
\end{definition}
Recall that the $\textit{polynomially convex hull}~\hat{K}$ and the $\textit{holomorphically convex hull}$ $\hat{K}_h$ of $K$ in $\Omega$ are defined as
\begin{gather}
\begin{align*}
	\hat{K}&:=\{z\in \Omega: |q(z)|\leq \|q\|_{K}~\text{for any}~q\in \mathbb{C}[z_1,\dots,z_n] \},~\text{and}\\
	\hat{K}_h&:=\{z\in \Omega: |f(z)|\leq \|f\|_{K}~\text{for any holomorphic function}~f:\Omega\to \mathbb{C}\},
\end{align*}
\end{gather}
respectively. Then $\hat{K}_h\subset \hat{K}\subset \hat{K}_\lambda$ for any compact set $K\subset \Omega$.

\begin{proposition}\label{polyhull}
If $K$ is a $\lambda$-circular compact subset of a $\lambda$-balanced domain $\Omega\subset \mathbb{C}^n$, then we have
\[
\hat{K}_h=\hat{K}=\hat{K}_\lambda.
\]
\end{proposition}
\begin{proof}
It suffices to show that $\hat{K}_{\lambda}\subset \hat{K}_{h}$. Choose a holomorphic function $f:\Omega\to \mathbb{C}$ and let $f=\sum q_m$ be the power series expansion of $f$ on $\Omega$ given by Theorem \ref{quasihomexp}. Choose $z\in K$. Since $\Omega$ is $\lambda$-balanced, the map $t\in \mathbb{H}\to f(\Phi^X(z,t))$ is well-defined for each $z\in K$. Furthermore, the map is also well-defined for any $t\in \partial{\mathbb{H}}$ as $K$ is $\lambda$-circular. Applying Proposition \ref{asympholo} to the bounded map
\[
t\in \bar{\mathbb{H}}\to f\circ \Phi^X(z,t)=\sum_{m=0}^{\infty}q_m(z)e^{-\rho_mt},
\]
we obtain 
\begin{equation}\label{leafwiseest}
|q_m(z)|\leq \|f\|_{\bar{L}_z}=\|f\|_{\partial \bar{L}_z}\leq \|f\|_{K}~\text{for each}~z\in  K,~m\geq 0
\end{equation}
so that $\|q_m\|_{K}\leq \|f\|_{K}$. Note that the equality in (\ref{leafwiseest}) follows from the maximum principle applied to $t\in \bar{\mathbb{H}}\to f\circ \Phi^X(z,t)$. Let $z\in \hat{K}_{\lambda}$. Then $|q_m(z)|\leq \|q_m\|_{K}$ for any $q_m\in \mathcal{H}_{\lambda}$ with $\textup{bideg}\,q_m=(\rho_m,0)$. So it follows from (\ref{leafwiseest}) and Lemma \ref{convergence} that
\begin{align}\label{hullest1}
|f(\Phi^X(z,t))|&= \bigg|\sum_{m=0}^{\infty}q_m(z)\cdot e^{-\rho_m\,t}\bigg|\leq \sum_{m=0}^{\infty}|q_m(z)|\cdot e^{-\rho_m \textup{Re}\,t}\\
&\leq \sum_{m=0}^{\infty} \|q_m\|_K\cdot e^{-\rho_m \textup{Re}\,t}\leq \|f\|_K\cdot\sum_{m=0}^{\infty} e^{-\rho_m \textup{Re}\,t}<+\infty  \nonumber
\end{align}
if $\textup{Re}\,t>0$. Fix a positive number $k\geq 1$ and replace $f$ in (\ref{hullest1}) with $f^k$. Then take the $k$th root of both sides of the inequality and let $k\to \infty$ to obtain $|f(\Phi^X(z,t))|\leq \|f\|_{K}$. Letting $t\to 0$, we obtain $|f(z)|\leq \|f\|_{K}$. Since $f$ is an arbitrary holomorphic function, $z\in \hat{K}_{h}$ so that $\hat{K}_{\lambda}\subset \hat{K}_{h}$ as desired.
\end{proof}

\begin{proposition}\label{approximation prop}
Let $u\in H_{\lambda}$ be a continuous function. If there exists a proper function $v:\mathbb{C}^n\to \mathbb{R}$ such that $u\geq v$ on $\mathbb{C}^n$, then
\begin{equation}\label{approximation 1}
u(z)=\sup\limits_{q_m\in \mathcal{H}_{\lambda}}\{|q_m(z)|^{\frac{1}{\rho_m}}: \textup{bideg}\,q_m=(\rho_m,0),\;|q_m|^{\frac{1}{\rho_m}}\leq u~\text{on}~\mathbb{C}^n\}
\end{equation}
for each $z\in \mathbb{C}^n$.
\end{proposition}
\begin{proof}
Denote by $\tilde{u}$ the function on the right-hand side of (\ref{approximation 1}). Then we immediately have $\tilde{u}\leq u$ on $\mathbb{C}^n$. To prove that $u\leq \tilde{u}$ on $\mathbb{C}^n$, it suffices to show that $\tilde{u}(a)\geq 1$ whenever $u(a)=1$ as the set of the integral curves of $(X,\lambda)$ forms a foliation of $\mathbb{C}^n-\{0\}$. Define a $\lambda$-balanced domain of holomorphy $\Omega:=\{z\in \mathbb{C}^n: u(z)<1\}$ containing the origin and a $\lambda$-circular set
\[
K^t:=\{z\in \mathbb{C}^n: u(z)\leq e^{-t}\}\subset\Omega
\]
for each $t>0$. Since $u$ is continuous and $v$ is proper, each $K^t$ is closed and bounded; that is, $K^t$ is compact in $\mathbb{C}^n$. The set $K^t$ is also relatively compact in $\Omega$ as $\textup{dist}(K^t,\partial \Omega)>0$. Then by Proposition \ref{polyhull} and the Cartan-Thullen theorem, $\hat{(K^t)}_{\lambda}=\hat{(K^t)}_{h}$ is a relatively compact subset of $\Omega$. Note that $a\in \partial \Omega$. So, for each $t\in (0,1)$, there exists a number $s\in (0,t)$ such that  $\Phi^X(a,s)\notin \hat{(K^t)}_{\lambda}$. This means that
\begin{equation}\label{polyest1}
1=\|q_m\|_{K^t}<|q_m(\Phi^X(a,s))| 
\end{equation}
for some $q_m\in H_{\lambda}.$ Note also that, if $z\in \partial \Omega$, then $u(z)=1$ and $\Phi^X(z,t)\in K_t.$ So the equality in (\ref{polyest1}) yields
\begin{equation*}\label{polyest2}
|q_m(\Phi^X(z,t))|^{\frac{1}{\rho_m}}=e^{-t}|q_m(z)|^{\frac{1}{\rho_m}}=1=u(z)~\text{if}~z\in \partial \Omega.
\end{equation*}
As $e^{-t}|q_m|^{\frac{1}{\rho_m}},\,u\in H_{\lambda}$ and the set of integral curves of $(X,\lambda)$ forms a foliation of $\mathbb{C}^n-\{0\}$, we have $e^{-t}|q_m(z)|^{\frac{1}{\rho_m}}=u(z)$ for any $z\in \mathbb{C}^n$. Then by the definition of $\tilde{u}$, $e^{-t}|q_m(z)|^{\frac{1}{\rho_m}}\leq \tilde{u}(z).$ Let $z=\Phi^X(a,s)$ in the inequality and use the inequality in (\ref{polyest1}) to  obtain $e^{-t}<e^{-s}\cdot \tilde{u}(a)$. If $t\to 0$, then $s\to 0$ so we conclude that $1\leq \tilde{u}(a)$.
\end{proof}
Now we present our main approximation theorem.
\begin{theorem}\label{mainapproximation}
	Given a function $u\in H_{\lambda}$, there exists a sequence $\{u_m\}\subset H_{\lambda}$ of continuous function on $\mathbb{C}^n$ satisfying $u(z)=\lim_{m\to \infty}u_m(z)~\text{for each}~z\in \mathbb{C}^n$ and
\begin{equation}\label{supremum}
	u_m(z)=\sup\limits_{q_k\in \mathcal{H}_{\lambda}}\{|q_k(z)|^{\frac{1}{\rho_k}}: \textup{bideg}\,q_k=(\rho_k,0),\,|q_k|^{\frac{1}{\rho_k}}\leq u_m~\text{on}~\mathbb{C}^n\}.
\end{equation}
\end{theorem}
\begin{proof}
Denote by $\mu$ the Lebesgue measure on $\mathbb{C}^n=\mathbb{R}^{2n}$ and by $(\cdot,\cdot)$ the function on $\mathbb{C}^n \times \mathbb{C}^n$ defined as
\[
(z,z'):=(z_1z'_1,\dots,z_nz'_n)\in \mathbb{C}^n,~z=(z_1,\dots,z_n),\,z'=(z'_1,\dots,z'_n)\in \mathbb{C}^n.
\]
Let $\omega:\mathbb{C}^n\to \mathbb{R}$ be a smooth function such that the support of $\omega$ is compact in $B^n$ and $\int_{\mathbb{C}^n}\omega(z)d\mu(z)=1$. Fix $u\in H_{\lambda}\subset L^1_{\textup{loc}}(\mathbb{R}^{2n}).$ For each positive integer $m\geq 1$ and $z=(z_1,\dots,z_n)\in \mathbb{C}^n$, define
\[
u_m(z):=\int_{\mathbb{C}^n}u(z+\frac{1}{m}(z,z'))\,\omega(z')\,d\mu(z')+\frac{1}{m}\cdot \sum_{k=1}^{n}|z_k|^{\frac{1}{\lambda_k}}.
\]
Then it follows from the standard smoothing arguments that each $u_m$ is continuous, $u_m\in H_{\lambda}$, and $\lim_{m\to \infty}u_m=u$ on $\mathbb{C}^n$. Since the inequality
\[
u_m(z)\geq \frac{1}{m}\cdot \sum_{k=1}^{n}|z_k|^{\frac{1}{\lambda_k}}:=v_m(z)~\text{for any}~z\in \mathbb{C}^n
\]
holds for each $m\geq 1$ and the function $v_m$ is proper, we obtain the desired conclusion by applying Proposition \ref{approximation prop} to each $u_m$.
\end{proof}

\subsection{$\lambda$-pluripolar sets}\label{subsection4.2}
 
\begin{definition}
\normalfont
A set $E\subset \mathbb{C}^n$ is $\textit{pluripolar}$ if there exists a nonconstant function $u\in \textup{PSH}(\mathbb{C}^n)$ such that $E \subset \{z\in \mathbb{C}^n: u(z)=-\infty\}$. Let $(X,\lambda)$ be a vector field on $\mathbb{C}^n$. A set $E\subset \mathbb{C}^n$ is called $\lambda$-$\textit{pluripolar}$ if there exists a function $u\in H_{\lambda}$ such that $E \subset \{z\in \mathbb{C}^n: u(z)=0\}$. 
\end{definition}

Note that a $\lambda$-pluripolar set is always pluripolar by Remark $\ref{remarklogofHlambdaftn}$. It turns out that a $\lambda$-circular pluripolar set is $\lambda$-pluripolar; see Theorem \ref{characterizationofHppsets}. 

The following lemma is fundamental for the arguments in this subsection.
\begin{lemma}\label{product}
Let $\{u_m\}\subset H_{\lambda}$ be a sequence satisfying $\|u_m\|_{B^n}\leq 1$ for each $m$ and define
\[
 u:=\prod_{m=1}^{\infty}(u_m)^{\frac{1}{2^m}}.
\]
Then $u\equiv 0$ or $u\in H_{\lambda}$.
\end{lemma}
\begin{proof}
For each $z\in \mathbb{C}^n-\{0\}$, define $z'_{\lambda}:=(z_1\cdot \|z\|^{-\frac{\lambda_1}{\textup{min}(\lambda)}},\dots,z_n\cdot \|z\|^{-\frac{\lambda_n}{\textup{min}(\lambda)}}).$ Then $z'_{\lambda}\in B^n$ whenever $\|z\|\geq 1$. Let $v\in H_{\lambda}$ and note that 
\[
v(z)=\|z\|^{\frac{1}{\textup{min}(\lambda)}}\cdot v(z'_{\lambda})\leq \|z\|^{\frac{1}{\textup{min}(\lambda)}}\cdot \|v\|_{S^{2n-1}}~\text{if}~\|z\|\geq 1.
\]
Therefore, we have
\begin{align}\label{Hlambdagrowth}
v(z)\leq \|v\|_{S^{2n-1}}\cdot \textup{max}\,(1,\|z\|^{\frac{1}{\textup{min}(\lambda)}})    =\|v\|_{B^n}\cdot \textup{max}\,(1,\|z\|^{\frac{1}{\textup{min}(\lambda)}}) 
\end{align}
for any $z\in \mathbb{C}^n$. Let $R>1$. Then by (\ref{Hlambdagrowth}), we have $
u_m(z)\cdot {R^{-\frac{1}{\textup{min}(\lambda)}}}\leq 1$
if $\|z\|\leq R$.  For each positive integer $\ell$, define
\[
v_\ell:=\prod_{m=1}^{\ell}(u_m\cdot {R^{-\frac{1}{\textup{min}(\lambda)}}})^{\frac{1}{2^m}}.
\]
 Note that $v_\ell\in \textup{PSH}(\mathbb{C}^n)$ by Remark \ref{remarklogofHlambdaftn}. Furthermore, the nonincreasing sequence $\{v_{\ell}\}$ is uniformly bounded on $B^n(0;R)$. So $v:=\lim_{\ell\to \infty}v_\ell=u\cdot R^{-\frac{1}{\textup{min}(\lambda)}}$  defines a plurisubharmonic function on $B^n(0;R)$. Since $R>0$ was arbitrary, we have $v\in H_{\lambda}$ if $v \not\equiv 0$ on $\mathbb{C}^n$. Then $u\equiv 0$ or $u\in H_{\lambda}$ as desired.
\end{proof}
\begin{proposition}\label{countablepp}
A countable union of $\lambda$-pluripolar sets is $\lambda$-pluripolar.
\end{proposition}
\begin{proof}
 Let $\{E_m:m\geq 1\}$ be a sequence of $\lambda$-pluripolar sets. Note that a finite union of $\lambda$-pluripolar sets is $\lambda$-pluripolar by Remark \ref{remarklogofHlambdaftn}. So by replacing each $E_m$ with $E_1\cup\cdots \cup E_m,$ we may assume that the given sequence is increasing. For each $m$, let  $u_m\in H_{\lambda}$ be a function such that $u_m\equiv 0$ on $E_m$. Then the sequence $\{u_m\}$ can be normalized so that $\|u_m\|_{B^n}=1$. By Lemma \ref{Hartogslemma}, there is a point $a\in \mathbb{C}^n$ such that $\limsup_{m\to \infty}u_m(a)\geq 1/2.$ Choose an increasing sequence $\{m_k\}$ of positive integers such that $u_{m_k}(a)> 1/2$ for any $k$. Then by Lemma \ref{product}, we have
\[
 u:=\prod\limits_{k=1}^{\infty}(u_{m_k})^{\frac{1}{2^k}}\in H_{\lambda}
\]
since $u(a)\geq 1/2.$ Note also that $u\equiv 0$ on $E:=\bigcup_{m=1}^{\infty}E_m$ as the sequence $\{E_m\}$ is increasing. Therefore, $E$ is $\lambda$-pluripolar.
\end{proof}

\begin{theorem}\label{pluripolarchar1}
Let $\{u_i\}_{i\in I}\subset H_{\lambda}$ be a given family and define
\[
u:=\sup\limits_{i\in I} u_i,\;S:=\{z\in \mathbb{C}^n:u(z)<+\infty\}.
\]
Then the following are equivalent.
\begin{enumerate}
\setlength\itemsep{0.05em}
\item $\|u\|_{B^n}<+\infty$.
\item $u^{\ast}\in H_{\lambda}$.
\item There exists a point $a\in \mathbb{C}^n$ such that $u^{\ast}(a)<+\infty$.
\item There exists an open ball $B^n(a;r)$ such that $\|u\|_{B^n(a;r)}<+\infty$.
\item $S$ is not $\lambda$-pluripolar.
\end{enumerate}
\end{theorem}
\begin{proof}
The implication $(1)\Longrightarrow(2)$ follows from (\ref{Hlambdagrowth}) and the implications $(2)\Longrightarrow(3)\Longrightarrow(4)\Longrightarrow(5)$ are obvious. We shall prove that $(5)$ does not hold if $(1)$ is false. Suppose that $\|u\|_{B^n}=+\infty$ and choose a subsequence $\{m_k\}\subset I$ such that $\|u_{m_k}\|_{B^n}\geq \textup{exp}\,(2^k)$. Then for each $k$, define
\[
v_k:=\frac{u_{m_k}}{\|u_{m_k}\|_{B^n}}\in H_{\lambda}
\] 
so that $\|v_k\|_{B^n}=1$. By Lemma \ref{Hartogslemma}, there exists a point $a\in \mathbb{C}^n$ and an increasing sequence $\{n_k\}$ of positive integers such that $v_{m_{n_k}}(a)\geq \frac{1}{2}$ and $n_k>2k$ for each $k$. Then by Lemma \ref{product}, we have $v:=\prod_{k=1}^{\infty}(v_{m_{n_k}})^{2^{-k}}\in H_{\lambda}$. If $z\in S$, then
\begin{align*}
v(z)\leq  u(z)\prod\limits_{k=1}^{\infty}\textup{exp}\,(-2^{n_k-k})\leq  u(z)\prod\limits_{k=1}^{\infty}\textup{exp}\,(-2^{k})=0.
\end{align*}
Therefore, $S$ is $\lambda$-pluripolar.
\end{proof}

	\subsection{Properties of the extremal function $\Psi_{E,\lambda}$ and the capacity $\rho_{\lambda}$}\label{subsection 4.3}
	\begin{theorem}\label{quasiextremal}
		For each compact set $K\subset \mathbb{C}^n$ and $z\in \mathbb{C}^n$,  we have 
		\[
		\Psi_{K,\lambda}(z)=\sup\,\{|q_m(z)|^{\frac{1}{\rho_m}}: q_m\in \mathcal{H}_{\lambda},\,\textup{bideg}\,q_m=(\rho_m,0),\,\|q_m\|_{K}\leq 1,\,m\geq1\}.
		\]
	\end{theorem}
	\begin{proof}
		Denote by $\hat{\Psi}_{K,\lambda}$ the function on the right-hand side of the equation above. Let $q_m\in \mathcal{H}_{\lambda}$ with $\textup{bideg}\,q_m=(\rho_m,0)$, $m\geq 1$. Then $|q_m|^{\frac{1}{\rho_m}}\in H_{\lambda}$ so we have $\hat{\Psi}_{K,\lambda}\leq \Psi_{K,\lambda}$ on $\mathbb{C}^n$. To prove that $\hat{\Psi}_{K,\lambda}\geq \Psi_{K,\lambda}$ on $\mathbb{C}^n$, fix  $u\in H_{\lambda}$ with $\|u\|_{K}\leq 1$. By Theorem \ref{mainapproximation}, there is a sequence $\{u_m\}\subset H_{\lambda}$ of continuous functions such that $u=\lim_{m\to \infty}u_m$ on $\mathbb{C}^n.$ Then it follows from Theorem \ref{pluripolarchar1} that $\{u_m\}$ is locally uniformly bounded on $\mathbb{C}^n$. Let $\epsilon>0$ be a positive number and note that the set 
		\[
		\Omega_{\epsilon}:=\{z\in \mathbb{C}^n:u(z)<1+\epsilon\}
		\] 
		is an open neighborhood of $K$. By Lemma \ref{Hartogslemma}, there is a positive integer $N=N(K,\epsilon)$ such that $u_m\leq 1+2\epsilon~\text{on}~K~\text{if}~m\geq N.$ Then by (\ref{supremum}), we have
		\[
		u_m(z) \leq (1+2\epsilon)\hat{\Psi}_{K,\lambda}(z)~\text{for each}~ z\in \mathbb{C}^n,\,m\geq N
		\]
		so that $u\leq (1+2\epsilon)\hat{\Psi}_{K,\lambda}$ on $\mathbb{C}^n$. Therefore, $\Psi_{K,\lambda}\leq \hat{\Psi}_{K,\lambda}$ on $\mathbb{C}^n$ as desired.
	\end{proof}
\begin{proposition}\label{increasingunion1}
Let $E\subset \mathbb{C}^n$ be a set and define $E_m:=E\cap B^n(0;m)$ for each positive integer $m$. Then $\{\Psi_{E_m,\lambda}\}$ decreases to  $\Psi_{E,\lambda}$ on $\mathbb{C}^n.$
\end{proposition}
In particular, Proposition \ref{increasingunion1} implies that $\Psi_{B,\lambda}\leq  \Psi_{A,\lambda}$ if $A\subset B$.
\begin{proof}
Note that each $E_m$ is bounded and $E_m\subset E_{m+1}$. So we have 
\[
\varphi:=\lim\limits_{m\to \infty}\Psi_{E_m,\lambda}\geq \Psi_{E,\lambda}~\text{on}~\mathbb{C}^n.
\]
For each bounded set $F\subset E$, there is an integer $m_0$ such that $F\subset E_m$ if $m\geq m_0$. Then $\Psi_{F,\lambda}\geq \varphi \geq \Psi_{E,\lambda}$ and $\varphi=\Psi_{E,\lambda}$. 
\end{proof}
\begin{theorem}\label{compactunion}
If $\Omega\subset \mathbb{C}^n$ is open and $\{K_m\}$ is an increasing sequence of compact subsets of $\Omega$ satisfying $K_m\subset \textup{int}\,K_{m+1}$ and $\Omega=\bigcup_{m=1}^{\infty}K_m$, then $\{\Psi^{\ast}_{K_m,\lambda}\}$ decreases to $\Psi_{\Omega,\lambda}$ on $\mathbb{C}^n$. Furthermore, $\Psi_{\Omega,\lambda}\in H_{\lambda}$ and
\[
\Psi_{\Omega,\lambda}=\textup{inf}\,\{\Psi^{\ast}_{K,\lambda}:K~\text{is a compact subset of }~\Omega\}.
\]
\end{theorem}

\begin{proof}
Note first that $\varphi:=\lim_{m\to \infty}\Psi^{\ast}_{K_m,\lambda}\geq \Psi_{\Omega,\lambda}$ on $\mathbb{C}^n$. Fix $a\in \Omega$ and choose positive numbers $m_0,r$ such that $B^n(a;r)\subset \textup{int}\,K_m$ if $m\geq m_0$. Then $(\ref{Hlambdagrowth})$ implies that  
\begin{equation}\label{psiest}
\Psi^{\ast}_{K_m,\lambda}(z)\leq \bigg(\frac{\|z-a\|}{r}\bigg)^{\frac{1}{\textup{min}(\lambda)}}<1 ~\text{if}~z\in B^n(a;r),~m\geq m_0.
\end{equation}
Therefore, $\varphi<1$ on $B^n(a;r)$. Since $a$ was arbitrary, we also have $\varphi\leq 1$ on $\Omega$. By Theorem \ref{pluripolarchar1}, $\varphi\in H_{\lambda}$ so that $\varphi\leq \Psi_{\Omega,\lambda}$ on $\mathbb{C}^n$. 
\end{proof}

\begin{proposition}\label{domainofholo1}
If $\Omega$ is a $\lambda$-balanced domain of holomorphy, then
\[
\Omega=\{z\in \mathbb{C}^n:\Psi_{\Omega,\lambda}(z)<1\}.
\]
\end{proposition}
\begin{proof}
By Theorem \ref{compactunion} and (\ref{psiest}), we have $\Omega\subset \hat{\Omega}:=\{z\in \mathbb{C}^n:\Psi_{\Omega,\lambda}(z)<1\}$. To prove that $\hat{\Omega}\subset \Omega$, let a $f$ be a holomorphic function on $\Omega$. By Theorem \ref{quasihomexp}, one can choose a sequence $\{q_m\}\in \mathcal{H}_{\lambda}$ such that $\textup{bideg}\,q_m=(\rho_m,0)$ for each $m$ and $f=\sum_{m=0}^{\infty}q_m$ on $\Omega$. Choose a compact subset $K\subset \Omega$ and $z\in \mathbb{C}^n$. Then it follows from Theorem \ref{quasiextremal} that 
\[
|q_m(z)|\leq \|q_m\|_{K}\cdot \{\Psi^{\ast}_{K,\lambda}(z)\}^{\rho_m}
\]
if $m\geq 1$. Note that the map $t\in \mathbb{H}\to f\circ \Phi^X(z,t)$ is well-defined for any $z\in K$ as $\Omega$ is $\lambda$-balanced. So $\{q_m\}$ is uniformly bounded on $K$ by Proposition \ref{asympholo} and
\[
\limsup_{m\to \infty}|q_m(z)|^{\frac{1}{\rho_m}}\leq \Psi^{\ast}_{K,\lambda}(z).
\]
This reduces to
\begin{equation}\label{domofholoext}
\limsup_{m\to \infty}|q_m(z)|^{\frac{1}{\rho_m}}\leq \Psi_{\Omega,\lambda}(z)
\end{equation}
 by Theorem \ref{compactunion}. Therefore, $f$ extends to a holomorphic function on a domain $\hat{\Omega}\supset \Omega$ by Lemma \ref{convergence}. Then we conclude that $\hat{\Omega}=\Omega$ as $\Omega$ is a domain of holomorphy.
\end{proof}
\begin{theorem}\label{smallestdomainofholo}
If $\Omega\subset \mathbb{C}^n$ is an open set, then 
\[
\hat{\Omega}:=\{z\in \mathbb{C}^n:\Psi_{\Omega,\lambda}(z)<1\}
\]
is the smallest $\lambda$-balanced domain of holomorphy containing $\Omega$. Furthermore, we have $\Psi_{\hat{\Omega},\lambda}=\Psi_{\Omega,\lambda}$ on $\mathbb{C}^n$.
\end{theorem}
\begin{proof}
Since $\Psi_{\Omega,\lambda}\in H_{\lambda}$ by Theorem \ref{compactunion}, $\hat{\Omega}$ is a $\lambda$-balanced domain of holomorphy containing $\Omega$. Let $G\supset \Omega$ be a $\lambda$-balanced domain of holomorphy. Then $\Psi_{G,\lambda}\leq \Psi_{\Omega,\lambda}$ and by Proposition \ref{domainofholo1}, we have
\[
\hat{\Omega}=\{z\in \mathbb{C}^n:\Psi_{\Omega,\lambda}(z)<1\}\subset \{z\in \mathbb{C}^n:\Psi_{G,\lambda}(z)<1\}=G.
\]
To prove that $\Psi_{\hat{\Omega},\lambda}=\Psi_{\Omega,\lambda}$, note first that
\[
\hat{\Omega}:=\{z\in \mathbb{C}^n:\Psi_{\Omega,\lambda}(z)<1\}=\{z\in \mathbb{C}^n:\Psi_{\hat{\Omega},\lambda}(z)<1\}
\]
by Proposition \ref{domainofholo1}. As $\Psi_{\hat{\Omega},\lambda},\Psi_{\Omega,\lambda}\in H_{\lambda}$, one can argue as in the proof of Theorem \ref{quasihomstar} that $\Psi_{\hat{\Omega},\lambda}$ $\equiv \Psi_{\Omega,\lambda}$ on $\mathbb{C}^n$.
\end{proof}
\begin{corollary}\label{domainofholoext}
If $E\subset \mathbb{C}^n$ is a bounded set, then
\[
\Psi_{E,\lambda}=\textup{sup}\,\{\Psi_{\Omega,\lambda}:\Omega~\text{is a}\; \lambda\text{-balanced open set containing}~E\}~\text{on}~\mathbb{C}^n.
\]
\end{corollary}
\begin{proof}
Let $A(z)$ be the function on the right-hand side of the equation above and note that $A\leq \Psi_{E,\lambda}$ on $\mathbb{C}^n$. Fix a point $z_0\in \mathbb{C}^n$ and a positive number $m<\Psi_{E,\lambda}(z_0)$. Then one can choose a function $u\in H_{\lambda}$ such that $\|u\|_{E}\leq 1$ and $u(z_0)>m$. For each positive number $\epsilon>0$, the set $\Omega_{\epsilon}:=\{z\in \mathbb{C}^n:u(z)<1+\epsilon\}$ is a $\lambda$-balanced open neighborhood of $E$ and
\[
m<u(z_0)\leq (1+\epsilon)\Psi_{\Omega_{\epsilon},\lambda}(z_0)\leq (1+\epsilon)A(z_0).
\]
Therefore, $\Psi_{E,\lambda}(z_0)\leq A(z_0)$. Since $z_0$ is arbitrary, we have $\Psi_{E,\lambda}\leq A$ on $\mathbb{C}^n$. 
\end{proof}

\begin{definition}[\cite{Siciak81}]\label{pluricomplexdef}
	\normalfont
	For each positive 
	integer $n$, let
	\[ 
	\mathcal{L}_n:=\{u\in \textup{PSH}(\mathbb{C}^n):\exists\,C_u\in \mathbb{R}
	~ \text{such that}~u(z)\leq C_u+\textup{log}\,(1+\|z\|)
	~ \forall z\in \mathbb{C}^n \}.
	\]  
	If $E$ is a bounded subset of $\mathbb{C}^n$, then define
	\[
	V_E(z):=\textup{sup}\,\{u(z)\colon u\in \mathcal{L}_n ,u\leq 0 ~\text{on}~ E\},\;~\forall z\in \mathbb{C}^n.
	\]
	If $E\subset \mathbb{C}^n$ is unbounded, then 
	\[
	V_{E}(z):=\textup{inf}\,\{V_{F}(z): F\subset E~\text{is bounded} \},\;~\forall z\in \mathbb{C}^n.
	\]	
	For any set $E\subset \mathbb{C}^n$, the function $V_E$ is called the $\textit{pluricomplex Green function}$ of $E$. We also define $\Phi_E:=\textup{exp}\,V_E$.
\end{definition}
\begin{remark}
	\normalfont
	It follows from Theorem 3.3 and Theorem 3.8 in \cite{Siciak82} that Proposition \ref{increasingunion1} and Theorem \ref{compactunion} also hold when each extremal function of the form $\Psi_{E,\lambda}$ is replaced by the function $\Phi_{E}$. One can also proceed as in the proof of Corollary \ref{domainofholoext} to obtain the equation
	\begin{equation}\label{remark equation}
	\Phi_{E}=\textup{sup}\,\{\Phi_{\Omega}:\Omega~\text{is a}\; \lambda\text{-circled open set containing}~E\}~\text{on}~\mathbb{C}^n
	\end{equation}
	if $E$ is a bounded subset of $\mathbb{C}^n$ and $(X,\lambda)$ is a vector field on $\mathbb{C}^n.$
	
	 For the proof, let $B(z)$ be the function on the right-hand side of the equation above and note that $B\leq \Phi_{E}$ on $\mathbb{C}^n$. Fix a point $z_0\in \mathbb{C}^n$ and a positive number $m<\Phi_{E}(z_0)$. Then one can choose a nonnegative function $u:\mathbb{C}^n\to \mathbb{R}$ such that  $\textup{log}\,u\in \mathcal{L}_n,\,\|u\|_{E}\leq 1$ and $u(z_0)>m$. Note that the set $\Omega_{\epsilon}:=\{z\in \mathbb{C}^n:u(z)<1+\epsilon\}$ is an open neighborhood of $E$ for each positive number $\epsilon>0$. So there exists a $\lambda$-circular neighborhood $\Omega'_{\epsilon}\subset \Omega_{\epsilon}$ of $E$ and
	\[
	m<u(z_0)\leq (1+\epsilon)\Phi_{\Omega'_{\epsilon}}(z_0)\leq (1+\epsilon)B(z_0).
	\]
	Therefore, $\Phi_{E}(z_0)\leq B(z_0)$. Since $z_0$ is arbitrary, we have $\Phi_{E}\leq B$ on $\mathbb{C}^n$. 
\end{remark}
It is known that $\Phi^{\ast}_E\equiv +\infty$ on $\mathbb{C}^n$ if, and only if, $E$ is pluripolar. If $E$ is a compact subset of $\mathbb{C}^n$, then it follows from Theorem 4.12 in \cite{Siciak81} that  
\begin{equation}\label{pluriGreenpolynomial}
	\Phi_E(z)=\textup{sup}\,\big\{|q(z)|^{\frac{1}{\textup{deg}\,q}}:\|q\|_{E}\leq 1,\, q\in \mathbb{C}[z_1,\dots,z_n] \big\}~\text{for each}~z\in \mathbb{C}^n.
\end{equation}

\begin{theorem}\label{plurigreenandextgeneral}
	If $E\subset \mathbb{C}^n$ is a $\lambda$-circular set, then we have
		\begin{equation}\label{plurigreenandext}
			\textup{max}\,(1,\Psi_{E,\lambda})^{\textup{min}(\lambda)}\leq {\Phi_E}\leq \textup{max}\,(1,\Psi_{E,\lambda})^{\textup{max}(\lambda)}
		\end{equation}
	on $\mathbb{C}^n$. In particular, a $\lambda$-circular set $E$ is pluripolar if, and only if, $\Psi^{\ast}_{E,\lambda}\equiv +\infty$.
\end{theorem}
\begin{proof}
	We first prove the inequalities when $E$ is compact. Choose a polynomial $q_m\in \mathcal{H}_{\lambda}$ with $\textup{bideg}\,q_m=(\rho_m,0)$, $m\geq 1$, $\|q_m\|_{E}\leq 1$. Then by (\ref{degreeandrho_k}) and (\ref{pluriGreenpolynomial}), we have
	\begin{align*}
		|q_m(z)|^{\frac{1}{\rho_m}}&= (|q_m(z)|^{\frac{1}{\textup{deg}\,q_m}})^{\frac{\textup{deg}\,q_m}{\rho_m}}\leq (\Phi_E(z))^{\frac{\textup{deg}\,q_m}{\rho_m}} \leq (\Phi_E(z))^{\frac{1}{\textup{min}(\lambda)}}
		\end{align*}
	for any $z\in \mathbb{C}^n$ since $\Phi_E\geq 1$ on $\mathbb{C}^n$. So Theorem \ref{quasiextremal} implies that $\textup{max}\,(1,\Psi_{E,\lambda})\leq (\Phi_E)^{\frac{1}{\textup{min}(\lambda)}}.$ 
	
	To prove the other inequality, let $q\in \mathbb{C}[z_1,\dots,z_n]$ be a polynomial with $\|q\|_{E}\leq 1$. Then choose finitely many polynomials $\{q_0,\dots,q_N\}\subset  \mathcal{H}_{\lambda} $ such that $\textup{bideg}\,q_m=(\rho_m,0)$ for each $m\in \{0,\dots,N\}$ and $q=\sum_{m=0}^{N}q_m$ on $\mathbb{C}^n$. Since $E$ is a $\lambda$-circular compact set, $|q(\Phi^X(z,t))|\leq 1$ for each $z\in E$ and $t\in \mathbb{C}$ with $\textup{Re}\,t=0$. So it follows from Proposition \ref{asympholo} that
	$\|q_m\|_{E}\leq 1$ for each $m$.  Fix $z\in \mathbb{C}^n$. By Theorem \ref{quasiextremal}, we have
	\[
	|q_m(z)|\leq (\Psi_{E,\lambda}(z))^{\rho_m}\leq \textup{max}\,(1,\Psi_{E,\lambda}(z))^{\rho_m}
	\]
	for each $m\geq 1$ so that
	\[
	|q(z)|\leq \sum_{m=0}^{N}|q_m(z)|\leq (N+1)\cdot \textup{max}\,(1,\Psi_{E,\lambda}(z))^{\rho_N}.
	\]
	Then
	\begin{align}\label{plurigreenest}
		|q|^{\frac{1}{\textup{deg}\,q}}&\leq (N+1)^{\frac{1}{\textup{deg}\,q}}\cdot \textup{max}\,(1,\Psi_{E,\lambda})^{{\frac{\rho_N}{\textup{deg}\,q}}}\\
		&\leq  (N+1)^{\frac{\textup{max}(\lambda)}{\rho_N}}\cdot \textup{max}\,(1,\Psi_{E,\lambda})^{\textup{max}(\lambda)}.\nonumber
	\end{align}
	Fix an integer $k\geq 1$ and replace $q$ in (\ref{plurigreenest}) by $q^k$. Since $\textup{bideg}\,q_N^k=(k\rho_N,0)$, we obtain
	\[
	|q|^{\frac{1}{\textup{deg}\,q}}\leq  (kN+1)^{\frac{\textup{max}(\lambda)}{k\rho_N}}\cdot \textup{max}\,(1,\Psi_{E,\lambda})^{\textup{max}(\lambda)}.
	\]
	Then letting $k\to \infty$ yields 
	\[
	|q|^{\frac{1}{\textup{deg}\,q}}\leq \textup{max}\,(1,\Psi_{E,\lambda})^{\textup{max}(\lambda)}.
	\]
     Therefore, it follows from (\ref{pluriGreenpolynomial}) that
	\[
	\Phi_E(z)\leq \textup{max}\,(1,\Psi_{E,\lambda}(z))^{\textup{max}(\lambda)}.
	\]
	This proves the claim when $E$ is compact.
	
	If the given set $E$ is a $\lambda$-circular open set, then (\ref{plurigreenandext}) follows from Theorem 3.8 in \cite{Siciak82} and Theorem \ref{compactunion}. If $E$ is $\lambda$-circular and bounded, then the formula also holds by Corollary \ref{domainofholoext} and (\ref{remark equation}). Finally, the general formula for any $\lambda$-circular unbounded set follows from Theorem 3.3 in \cite{Siciak82} and Proposition \ref{increasingunion1}.
\end{proof}
By Theorem 3.6 in \cite{Siciak82}, the polynomially convex hull $\hat{K}$ of a compact set $K$ in $\mathbb{C}^n$ is $\hat{K}=\{z\in \mathbb{C}^n: \Phi_K(z)\leq 1\}$. Then Proposition \ref{polyhull} and Theorem \ref{plurigreenandextgeneral} immediately yield the following
\begin{theorem}\label{polyhullmain}
	If $K$ is a $\lambda$-circular compact set in $\mathbb{C}^n$, then
	\[
    \hat{K}=\hat{K}_{\lambda}=\{z\in \mathbb{C}^n: \Psi_{K,\lambda}(z)\leq 1\}\ni 0.
	\]
\end{theorem}

Now we prove the main theorem of this section.
\begin{theorem}\label{characterizationofHppsets}
If $E$ is a nonempty subset of $\mathbb{C}^n$, then the following are equivalent.
\begin{enumerate}
\item $E$ is $\lambda$-pluripolar.
\item $\Psi^{\ast}_{E,\lambda}\equiv +\infty$.
\item $\rho_{\lambda}(E):=\textup{inf}\,\{\|u\|_{E}:u\in H_{\lambda},\, \|u\|_{S^{2n-1}}=1\}=0$.
\end{enumerate}
If $E$ is $\lambda$-circular, then any one of the statements above holds if, and only if, $E$ is pluripolar.
\end{theorem}
\begin{proof}
$(3)\Longrightarrow(1)$: Suppose first that $\rho_{\lambda}(E)=0$. Then there exists a sequence $\{u_m\}\subset H_{\lambda}$ such that $\|u_m\|_{S^{2n-1}}=1$ for each $m\geq 1$, and $\lim_{m\to \infty}\|u_m\|_{E}=0$.
Then by Lemma \ref{Hartogslemma}, there exists a point $a\in \mathbb{C}^n$ and a subsequence $\{u_{n_m}\}$ of $\{u_m\}$ such that $u_{n_m}(a)>1/2$ for each $m\geq 1$. Choose a subsequence $\{u_{n_{j_m}}\}$ of $\{u_{n_m}\}$ such that $\|u_{n_{j_m}}\|^{2^{-m}}_{E}\leq \frac{1}{2}~\text{for each}~m\geq 1.$
Then
\[
u:=\prod_{m=1}^{\infty}(u_{n_{j_m}})^{2^{-m}}\in H_{\lambda}
\]
by Lemma \ref{product}, and $u\equiv 0$ on $E$. Therefore, $E$ is $\lambda$-pluripolar.

 $(2)\Longrightarrow(3)$: We first show that 
 \begin{equation}\label{capacityandextr}
  \rho_{\lambda}(E)=\frac{1}{\|\Psi_{E,\lambda}\|_{S^{2n-1}}}=\frac{1}{\|\Psi^{\ast}_{E,\lambda}\|_{S^{2n-1}}}
 \end{equation}
 for any bounded set $E\subset \mathbb{C}^n$. Choose $u\in H_{\lambda}$ with $\|u\|_{S^{2n-1}}=1$. Then $u\leq \|u\|_{E}\cdot \Psi_{E,\lambda}$ on $\mathbb{C}^n$ so that
 \[
 1=\|u\|_{S^{2n-1}}\leq \|u\|_{E}\cdot \|\Psi_{E,\lambda}\|_{S^{2n-1}}.
 \]
 This implies that $ \rho_{\lambda}(E)\geq \frac{1}{\|\Psi_{E,\lambda}\|_{S^{2n-1}}}$. Fix a positive number $m<\|\Psi_{E,\lambda}\|_{S^{2n-1}}$ and choose $u\in H_{\lambda}$ such that $\|u\|_{E}\leq 1$, $\|u\|_{S^{2n-1}}>m$. Then $v:=\frac{u}{\|u\|_{S^{2n-1}}}\in H_{\lambda}$ and  $\|v\|_{S^{2n-1}}=1$. Therefore,
 \[
 \rho_{\lambda}(E)\leq \|v\|_{E}=\frac{1}{\|u\|_{S^{2n-1}}}<\frac{1}{m}
 \]
 so that  $ \rho_{\lambda}(E)\leq \frac{1}{\|\Psi_{E,\lambda}\|_{S^{2n-1}}}$. Now the formula
  \[
  \rho_{\lambda}(E)=\frac{1}{\|\Psi_{E,\lambda}\|_{S^{2n-1}}}\geq \frac{1}{\|\Psi^{\ast}_{E,\lambda}\|_{S^{2n-1}}}
 \]
 is obvious and (\ref{Hlambdagrowth}) implies (\ref{capacityandextr}). This proves the claim when $E$ is bounded. If $E$ is unbounded, then let $E_m=E\cap B^n(0;m)$ for each $m\geq 1$. The given assumption implies that $+\infty\equiv \Psi^{\ast}_{E,\lambda}\leq \Psi^{\ast}_{E_m,\lambda} $ so that $\rho_{\lambda}(E_m)=0$ for each $m$ by (\ref{capacityandextr}). Then by the implication $(3)\Longrightarrow(1)$, each $E_m$ is $\lambda$-pluripolar. Therefore, $E$ is $\lambda$-pluripolar by Proposition \ref{countablepp}. Hence $\rho_{\lambda}(E)=0$.
 
 $(1)\Longrightarrow(2)$: Since $E$ is assumed to be $\lambda$-pluripolar, there is a function $u\in H_{\lambda}$ such that $u\equiv 0$ on $E$. Then the conclusion follows from the inequality $m\cdot u\leq \Psi^{\ast}_{E,\lambda}$ on $\mathbb{C}^n$ for each $m>0$.

 Note that the last statement of the theorem follows from Theorem \ref{plurigreenandextgeneral}.
\end{proof}

Finally, we prove that the function $E\to \Psi^{\ast}_{E,\lambda}$ is continuous from above (Theorem \ref{increasingunion2}). The continuity will be important in the proof of Theorem \ref{normalsuspension}.
\begin{definition}
\normalfont
A property is said to hold $H_{\lambda}$-$\textit{almost everywhere}$ or $H_{\lambda}$-$\textit{a.e.}$ on a set $E\subset \mathbb{C}^n$ if it holds on $E-A$ for some $\lambda$-pluripolar set $A\subset \mathbb{C}^n$. 
\end{definition}
\begin{lemma}\label{Hlambdaa.e.property}
Let $\mathcal{F}\subset H_{\lambda}$ be a given family and define $u:=\textup{sup}\,\{v:v\in \mathcal{F}\}$. Then $u\equiv u^{\ast}$ $H_{\lambda}$-a.e. on $\mathbb{C}^n$
\end{lemma}
\begin{proof}
It suffices to show that $\mathcal{N}:=\{z\in \mathbb{C}^n:u(z)<u^{\ast}(z)\}$ is $\lambda$-pluripolar. If $A:=\{z\in \mathbb{C}^n:u(z)<+\infty\}$ is $\lambda$-pluripolar, then $u^{\ast}\equiv +\infty$ by Theorem \ref{pluripolarchar1}. So $\mathcal{N}=A$ is pluripolar. Note also that  $u^{\ast}\in H_{\lambda}$ if $A$ is not $\lambda$-pluripolar. Then by Theorem 7.1 of \cite{BedfordandTaylor82}, $\mathcal{N}$ is a $\lambda$-circular pluripolar set. Therefore, we conclude from Theorem \ref{characterizationofHppsets} that $\mathcal{N}$ is $\lambda$-pluripolar. 
\end{proof}
\begin{lemma}\label{Hlambdaae}
If $E$ is a subset of $\mathbb{C}^n$, then $\Psi^{\ast}_{E,\lambda}=\Psi_{E,\lambda}$ $H_{\lambda}$-a.e. on $\mathbb{C}^n$ and $\Psi^{\ast}_{E,\lambda}\leq 1$ $H_{\lambda}$-a.e. on $E$.
\end{lemma}
\begin{proof}
The statement follows from Lemma $\ref{Hlambdaa.e.property}$ if $E$ is bounded. Suppose that $E$ is unbounded and let $E_m:=E\cap B^n(0;m)$ for each positive integer $m$. Then by Lemma \ref{Hlambdaa.e.property}, we have
\begin{equation}\label{equation}
\Psi^{\ast}_{E_m,\lambda}=\Psi_{E_m,\lambda}~\text{on}~\mathbb{C}^n-A_m,\,\text{and}~\Psi^{\ast}_{E_m,\lambda}\leq 1~\text{on}~E-A_m
\end{equation}
where $A_m=\{z\in \mathbb{C}^n:u_m(z)=0\}$ for some $u_m\in H_{\lambda}$. Without loss of generality, we may assume that $E$ and each $E_m$ are not $\lambda$-pluripolar. Then it follows from Theorem $\ref{pluripolarchar1}$ that $\Psi^{\ast}_{E,\lambda},\,\Psi^{\ast}_{E_m,\lambda}\in H_{\lambda}$. Since $\{\Psi^{\ast}_{E_m,\lambda}\}$ is a decreasing sequence of plurisubharmonic functions, $\varphi:=\lim_{m\to \infty}\Psi^{\ast}_{E_m,\lambda}$ is plurisubharmonic. By Proposition \ref{increasingunion1} and (\ref{equation}), we have $\varphi=\Psi^{\ast}_{E,\lambda}$ on $\mathbb{C}^n-A$, where $A:=\bigcup_{m=1}^{\infty}A_m.$ Then note that $A$ is $\lambda$-pluripolar by Proposition \ref{countablepp} and therefore it is of $2n$-dimensional Lebesgue measure zero. So $\varphi=\Psi^{\ast}_{E,\lambda}$ on $\mathbb{C}^n$ as $\varphi,\,\Psi^{\ast}_{E,\lambda}\in H_{\lambda}$. Now we obtain the desired conclusion from (\ref{equation}).
\end{proof}

\begin{theorem}\label{extremala.e.}
If $E$ is a subset of $\mathbb{C}^n$, then
\[
\Psi^{\ast}_{E,\lambda}=\textup{sup}\,\{u:u\in H_{\lambda},\, u\leq 1\; H_{\lambda}\text{-a.e.}\, \text{on}\; E\}~\text{on}~\mathbb{C}^n.
\]
\end{theorem}
\begin{proof}
Denote by $\varphi$ the function on the right-hand side of the equation above. Choose a function $u\in H_{\lambda}$ satisfying $u\leq 1$ on $E-A$ with $A=\{z\in \mathbb{C}^n: v(z)=0\},\,v\in H_{\lambda}$. Then by Remark \ref{remarklogofHlambdaftn}, $u^{1-\epsilon}\cdot v^{\epsilon}\in H_{\lambda}$ for each $\epsilon\in (0,1)$. Note that, for any bounded subset $F$ of $E$, we have
\begin{equation}\label{equation2}
u^{1-\epsilon}\cdot v^{\epsilon}\leq \|u^{1-\epsilon}\cdot v^{\epsilon}\|_{F}\cdot \Psi_{F,\lambda} \leq\|v\|^{\epsilon}_{F}\cdot \Psi_{F,\lambda}~\text{on}~\mathbb{C}^n.
\end{equation}
  Let $z\in \mathbb{C}^n$ be a point such that $v(z)\neq 0$. Then letting $\epsilon \to 0$ in (\ref{equation2}), we obtain $u(z)\leq \Psi_{E,\lambda}(z)$ from the definition of $\Psi_{E,\lambda}$. So $u\leq \Psi^{\ast}_{E,\lambda}$  and $\varphi\leq \Psi^{\ast}_{E,\lambda}$ on $\mathbb{C}^n$.

If $E$ is $\lambda$-pluripolar, then $\varphi \equiv \Psi^{\ast}_{E,\lambda}\equiv +\infty$ by the implication $(1)\Longrightarrow(2)$  in Theorem \ref{characterizationofHppsets}. Suppose that $E$ is not $\lambda$-pluripolar. Then it follows from Theorem \ref{pluripolarchar1} and Lemma \ref{Hlambdaae} that  $\Psi^{\ast}_{E,\lambda}\in H_{\lambda}$ and $\Psi^{\ast}_{E,\lambda}\leq 1$ $H_{\lambda}$-a.e.\;on $E$, respectively. Therefore, we have $\Psi^{\ast}_{E,\lambda}\leq \varphi$ on $\mathbb{C}^n$.
\end{proof}
\begin{theorem}\label{increasingunion2}
If $E\subset \mathbb{C}^n$ is an increasing union of the sequence $\{E_m\}$, then
\[
\lim_{m\to \infty}\Psi^{\ast}_{E_m,\lambda}(z)=\Psi^{\ast}_{E,\lambda}(z)~\text{for any}~ z\in \mathbb{C}^n.
\]
\end{theorem}
\begin{proof}
If $E$ is $\lambda$-pluripolar, then we have $\Psi^{\ast}_{E_m,\lambda}=\Psi^{\ast}_{E,\lambda}\equiv +\infty$ for each $m\geq 1$ by Theorem \ref{characterizationofHppsets}. If $E$ is not $\lambda$-pluripolar, then $E_{m_0}$ is not $\lambda$-pluripolar for some $m_0\geq 1$ by Proposition \ref{countablepp}. Note also that $\Psi^{\ast}_{E_m,\lambda}\in H_{\lambda}$ if $m\geq m_0$ by Theorem \ref{pluripolarchar1} and Theorem \ref{characterizationofHppsets}. So $\varphi:=\lim_{m\to \infty}\Psi^{\ast}_{E_m,\lambda}\in H_{\lambda}$ and $\varphi\geq \Psi^{\ast}_{E,\lambda}$ on $\mathbb{C}^n$. Then by Lemma \ref{Hlambdaae}, $\varphi\leq 1$ $H_{\lambda}$-a.e. on $E$. Finally, we conclude from Theorem \ref{extremala.e.} that $\varphi\leq \Psi^{\ast}_{E,\lambda}$ on $\mathbb{C}^n$.
\end{proof}
\section{Normal suspensions}\label{Section normal}
\begin{definition}\label{Lregular}
\normalfont
A set $E\subset \mathbb{C}^n$ is $\textit{L-regular}$ at $a\in \bar{E}$ if $V^{\ast}_{E}(a)=0$. $E$ is said to be $\textit{locally L-regular}$ at $a\in \bar{E}$ if $E\cap B^n(a;r)$ is $L$-regular for each $r>0$. A set $E$ is $\textit{locally pluripolar}$ if, for each $z\in E$, there is an open neighborhood $U\subset \mathbb{C}^n$ of $z$ and a nonconstant function $u\in \textup{PSH}(U)$ such that $u\equiv -\infty$ on $E\cap U$.
\end{definition}

\begin{remark}\label{L-regular}
\normalfont
 It is known that a set is nonpluripolar if, and only if, it is locally $L$-regular at some point; if $E\subset \mathbb{C}^n$ is locally $L$-regular at $a\in \bar{E}$, then it is nonpluripolar as $V^{\ast}_{E}\equiv +\infty$ on $\mathbb{C}^n$ whenever $E$ is pluripolar. The converse follows from the fact that the set $\{z\in \bar{E}:E~\text{is not locally} ~L\text{-regular at}~z\}$ is always pluripolar. See p.186 of \cite{Klimek91}.  We also remark that $E$ is pluripolar if, and only if, $E$ is locally pluripolar by \cite{Josephson1978}.
\end{remark}

\begin{proposition}\label{regularnonpp}
	A suspension $S^X_0(F)$ has a regular leaf if, and only if, $S^X_0(F)$ is nonpluripolar.
\end{proposition}
\begin{proof}
	By the previous discussion, it suffices to show that $S^X_0(F)$ is pluripolar if, and only if, $F'_{\lambda}$ is pluripolar. Suppose that $S^X_0(F)$ is pluripolar. Then by Theorem \ref{characterizationofHppsets}, there exists a function $u\in H_{\lambda}$ such that $u\equiv 0$ on $S^X_0(F)$. Recall the notations in Definition \ref{def.susp.} and note that $u(z)=|z_1|\cdot u(1,z')=0$ whenever $z=(z_1,\dots,z_n)\in F$ and
	\[
	z'=\bigg(\frac{z_2}{z^{\lambda_2}_1},\dots,\frac{z_n}{z^{\lambda_n}_1}\bigg)\in F'_{\lambda,i},~z_1\notin C_i.
	\]
    If the map $v(z'):=u(1,z')\in \textup{PSH}(\mathbb{C}^{n-1})$ is constant on $\mathbb{C}^{n-1}$, then $S^X_0(F)\subset \{z\in \mathbb{C}^n:z_1=0\}$ so that $F'_{\lambda}=\emptyset$ is pluripolar. If $v$ is not constant, then it follows from (\ref{quasihom}) that $\textup{log}\,v$ is a nonconstant plurisubharmonic function on $\mathbb{C}^{n-1}$ such that $\textup{log}\,v\equiv -\infty$ on $F'_{\lambda}$. Therefore, we conclude that $F'_{\lambda}$ is pluripolar.
    
    Conversely, suppose that $F'_{\lambda}$ is pluripolar. To show that $S^X_0(F)$ is also pluripolar, consider the holomorphic map $\psi: \mathbb{C}^{n-1}\times \mathbb{C} \to \mathbb{C}^n$ defined as
   	\begin{gather*}
		\psi(z'_1,\dots,z'_{n-1},t)=(e^{-\lambda_1t},z'_1e^{-\lambda_2t},\dots,z'_{n-1}e^{-\lambda_nt}).
	\end{gather*}
	By a straightforward computation, one can show that the modulus of the determinant of the complex Jacobian of $\psi$ at a fixed point $(z',t)\in  \mathbb{C}^{n-1}\times \mathbb{C}$ is
	\[
	\textup{exp}\,{\Bigg(-\textup{Re}\,t\cdot \sum\limits_{k=1}^{n}\lambda_k\Bigg)}\neq 0.
	\]
	So the map $\psi$ is a local biholomorphism by the inverse function theorem. Let $w\in A:=S^X_0(F)\cap \{z\in \mathbb{C}^n:z_1\neq 0\}$. Then there is a number $i\in \{1,2\}$ and an open neighborhood $U\subset \mathbb{C}^n$ of $w$ such that 
	\begin{enumerate}
		\setlength\itemsep{0.1em}
		\item $z_1 \notin C_i~\text{if}~z=(z_1,\dots,z_n)\in U,$ 
		\item $\psi^{-1}|_{U}$ is a well-defined biholomorphism, and
		\item $\psi^{-1}(U\cap A)\subset F'_{\lambda,i}\times S_{r,s}$ where $S_{r,s}=\{t\in \mathbb{C}: r<\textup{Re}\,t<s\}$ for some $0<r<s$.
	\end{enumerate}
    Since $F'_{\lambda,i}\times S_{r,s}$ is also pluripolar by the given assumption, there is a nonconstant function $u\in \textup{PSH}(\mathbb{C}^n)$ such that $u\equiv -\infty$ on $F'_{\lambda,i}\times S_{r,s}$. Then $v:=u\circ \psi^{-1}|_{U}\in \textup{PSH}(U)$ is a nonconstant function satisfying $v\equiv -\infty$  on $U\cap A$. Therefore, $A$ is locally pluripolar and it is pluripolar. Since the set equality
	\[
	S^X_0(F)=A\cup (S^X_0(F)\cap \{z\in \mathbb{C}^n:z_1= 0\})
	\]
	holds and the set $\{z\in \mathbb{C}^n:z_1=0\}$ is pluripolar, we conclude that $S^X_0(F)$ is pluripolar.
\end{proof}
\begin{remark}\label{uniformIFT}
	\normalfont 
	Choose an open set $U\subset \mathbb{C}^{n-1},\,s>0$ and consider the restriction of $\psi$ to $U\times S_{0,s}$. Then the second-order partial derivatives of the map are uniformly bounded and the modulus of the complex Jacobian of the map is bounded from below by a positive uniform constant. So by a version of the inverse function theorem in \cite{Christ85}, there is a uniform number $R>0$ such that each point $w\in \psi(U\times S_{0,s})$ has an open neighborhood $B^n(w;R)$ on which $\psi^{-1}$ is a well-defined local biholomorphism.
\end{remark}

Choose a polynomial $q_m\in \mathcal{H}_{\lambda}$ with $\textup{bideg}\,q_m=(\rho_m,0),\,m\geq 1$ and a set $E\subset \mathbb{C}^n$. Then recall that Definition \ref{definitionofextrftn} implies the following Bernstein-Walsh type inequality:
\begin{equation}\label{BW inequality}
|q_m(z)|\leq \|q_m\|_{E}\cdot \{\Psi_{E,\lambda}(z)\}^{\rho_m}~\text{for any}~z\in \mathbb{C}^n.
\end{equation} 
\begin{theorem}\label{normalsuspension}
	Let $S^X_0(F)\subset \mathbb{C}^n$ be a regular suspension. If a formal series $S\in \mathbb{C}[[z_1,\dots,z_n]]$ is holomorphic along $S^X_0(F)$, then it is holomorphic on a domain of holomorphy
	\begin{equation}\label{set inclusion}
		\Omega:=\{z\in \mathbb{C}^n:\Psi^{\ast}_{S^X_0(F),\lambda}(z)<1\}\supset B^n(0;\{\rho_{\lambda}(S^{X}_0(F))\}^{\textup{max}(\lambda)})
	\end{equation}
	containing the origin. 	Conversely, Let $F\subset S^{2n-1}$ be a $\lambda$-circular $F_{\sigma}$ set such that $S^X_0(F)$ is not regular. Then there exists a formal power series $S\in \mathbb{C}[[z_1,\dots,z_n]]$ such that the series is holomorphic along $S^X_0(F)$ but it does not converge uniformly on any open neighborhood of the origin.
\end{theorem}
\begin{proof}
Suppose that $S^X_0(F)$ is a regular suspension. Then $\rho_{\lambda}(S^X_0(F))>0$ and the suspension is not $\lambda$-pluripolar by Theorem \ref{characterizationofHppsets} and Proposition \ref{regularnonpp}. Let $S\in \mathbb{C}[[z_1,\dots,z_n]]$ be a formal power series such that $t\in \mathbb{H}\to S(\Phi^{X}(z,t))$ is holomorphic for each $z\in F$. Then we are to show that $S$ converges uniformly on each compact subset of $\Omega$. Write $S=\sum_{m=0}^{\infty}q_m$, where $q_m\in H_{\lambda}$, and $\textup{bideg}\,q_m=(\rho_m,0)$. For each positive integer $k$ and $\ell$, define
\begin{align*}
	F_{k,\ell}&:=\Big\{z\in F: |S(\Phi^{X}(z,t))|\leq k~\text{if}~\textup{Re}\,t>\frac{1}{\ell}\Big\},\\
	L_{k,\ell}&:=\Big\{\Phi^{X}(z,t): z\in F_{k,\ell},~\textup{Re}\,t>\frac{1}{\ell}\Big\}.
\end{align*}
Since $S$ is holomorphic along $S^X_0(F)$, we have
\begin{equation}\label{union}
 \bigcup\limits_{k,\ell=1}^{\infty}L_{k,\ell}=S^X_0(F).
\end{equation}
Fix $z\in F_{k,\ell}$ and $t\in \mathbb{C}$ with $\textup{Re}\,t>\frac{1}{\ell}$. Then
\[
S(\Phi^{X}(z,t))=\sum_{m=0}^{\infty}q_m(z)e^{-\rho_m t}=\sum_{m=0}^{\infty}q_m(z)e^{-\frac{\rho_m}{\ell}}e^{-\rho_m (t-\frac{1}{\ell})}.
\]
By Proposition \ref{asympholo}, we have
\[
|q_m(z)|\leq k\cdot e^{\frac{\rho_m}{\ell}}~\text{for each}~m.
\]
So $|q_m(z)|\leq k$ for each $m\geq 1$ and $z\in L_{k,\ell}$. Then the estimate
\[
|q_m(z)|\leq k\cdot (\Psi^{\ast}_{L_{k,\ell},\lambda}(z))^{\rho_m}~\text{for each}~z\in \mathbb{C}^n,\, k,\ell,m\geq 1
\]
 follows from (\ref{BW inequality}) so that
\begin{equation}\label{bwinequality}
\limsup_{m\to \infty}|q_m(z)|^{\frac{1}{\rho_m}}\leq \Psi^{\ast}_{L_{k,\ell},\lambda}(z)~\text{for any}~z\in \mathbb{C}^n,\, k,\ell\geq 1.
\end{equation}
By Theorem \ref{increasingunion2} and (\ref{union}), (\ref{bwinequality}) reduces to 
\begin{equation*}
\limsup_{m\to \infty}|q_m(z)|^{\frac{1}{\rho_m}}\leq \Psi^{\ast}_{{S^X_0(F)},\lambda}(z)~\text{for each}~ z\in \mathbb{C}^n.
\end{equation*}
 Since $S^X_0(F)$ is nonpluripolar, $\Psi^{\ast}_{S^X_0(F),\lambda}\in H_{\lambda}$  by Theorem \ref{pluripolarchar1} and Theorem \ref{characterizationofHppsets}.
Therefore, $\Omega$ is a domain of holomorphy containing the origin. Furthermore, it follows from Lemma \ref{convergence} that $S$ converges uniformly on each compact subset of  $\Omega$. So the series $S$ is holomorphic on $\Omega$.

 Now we establish the set inclusion (\ref{set inclusion}). Note first that $|z_1|\leq \Psi^{\ast}_{B^n}(z)\leq 1$ for each $z\in \mathbb{C}^n$. Then $\rho_{\lambda}(B^n)=1$ by (\ref{capacityandextr})  and we also have $0\neq \rho_{\lambda}(S^X_0(F))\leq \rho_{\lambda}(B^n)=1$ for any $(X,\lambda)$ and $F\subset S^{2n-1}$. Suppose that 
 \[
 \|z\|< \{\rho_{\lambda}(S^X_0(F))\}^{\textup{max}(\lambda)}\leq 1,\, z\neq 0
 \] 
 and let 
\begin{equation}\label{z_lambda}
z_{\lambda}:=(z_1\cdot \|z\|^{-\frac{\lambda_1}{\textup{max}(\lambda)}},\dots,z_n\cdot \|z\|^{-\frac{\lambda_n}{\textup{max}(\lambda)}})\in B^n.
\end{equation}
Since $\Psi^{\ast}_{S^X_0(F)}\in H_{\lambda}$, we have
\begin{align*}
 \Psi^{\ast}_{S^X_0(F)}(z)&=\|z\|^{\frac{1}{\textup{max}(\lambda)}}\cdot \Psi^{\ast}_{S^X_0(F)}(z_{\lambda})\leq \|z\|^{\frac{1}{\textup{max}(\lambda)}}\cdot \big\|\Psi^{\ast}_{S^X_0(F)}\big\|_{S^{2n-1}}\\
 &=  \|z\|^{\frac{1}{\textup{max}(\lambda)}}\cdot \{\rho_{\lambda}(S^X_0(F))\}^{-1}<1
\end{align*}
by (\ref{capacityandextr}). Therefore, $z\in \Omega$ as desired.

Conversely, suppose that $S^X_0(F)$ is not regular and $F\subset S^{2n-1}$ is a $\lambda$-circular $F_{\sigma}$ set. Then one can choose an increasing sequence $\{K_m\}$ of compact subset of $S^{2n-1}$ such that $F=\bigcup_{m=1}^{\infty} K_m$. Since $S^X_0(F)$ is $\lambda$-pluripolar by Theorem \ref{characterizationofHppsets} and Proposition \ref{regularnonpp}, there exists a function $u\in H_{\lambda}$ such that $u\equiv 0$ on $F$. Note that $u$ can be written as
\[
u= {\bigg(\limsup\limits_{m\to \infty}|p_m|^{\frac{1}{\rho_m}}{\bigg)}}^{\ast},
\]
where $p_m\in \mathcal{H}_{\lambda},$ and $\textup{bideg}\,p_m=(\rho_m,0)$ for each $m$ by Theorem \ref{quasihomstar}. As $u$ is nonconstant, there is a point $a\in \mathbb{C}^n$ such that $\limsup_{m\to \infty}|p_m(a)|^{\frac{1}{\rho_m}}\neq 0$. Choose an increasing sequence $\{n_m\}$ of positive integers and $A>0$ such that $|p_{n_m}(a)|\geq A^{\rho_{n_m}}>0$ for each $m\geq 1$. Since $u(z)=0$ for each $z\in F$,  one can apply Lemma \ref{Hartogslemma} and assume that 
\[
|p_{n_m}(z)|^{\frac{1}{\rho_{n_m}}}\leq \frac{1}{m^2}~\text{for any}~z\in K_m,~m\geq 1
\]
by taking a subsequence of $\{n_m\}$ if necessary. For each $m\geq 1,$ define a polynomial
\[
q_m(z):= m^{\rho_{n_m}}\frac{p_{n_m}(z)}{p_{n_m}(a)}\in \mathcal{H}_{\lambda}
\]
with $\textup{bideg}\,q_m=(\rho_m,0)$. Note that 
\[
(\|q_m\|_{K_{m_0}})^{\frac{1}{\rho_{n_m}}}\leq (m\cdot|p_{n_m}(a)|^{\frac{1}{\rho_{n_m}}})^{-1}\leq \frac{1}{A\cdot m}	
\]
whenever $m\geq m_0$. Therefore, the formal series $S:=\sum_{m=1}^{\infty}q_m\in \mathbb{C}[[z_1,\dots,z_n]]$ is uniformly convergent on each $K_m$ by Lemma \ref{convergence}. For each $k\geq 1$, define
\[
b_k:=\bigg(\frac{a_1}{k^{\lambda_1}},\dots,\frac{a_n}{k^{\lambda_n}}\bigg).
\]
Then $\lim_{k\to \infty}b_k=0\in \mathbb{C}^n$ and $q_m(b_k)=\big(\frac{m}{k}\big)^{\rho_{n_m}}~\text{for any}~k,m\geq 1$ so that $S(b_k)=\sum_{m=1}^{\infty}q_m(b_k)$ is divergent for each $k$. Hence, the correspondence $t\in \mathbb{H}\to S\circ \Phi^{X}(z,t)$ defines a holomorphic function for each $z\in F$ but $S$ does not converge uniformly on any open neighborhood of the origin. 
\end{proof}

Let $(X,\lambda)$ be a vector field on $\mathbb{C}^n$ and $F\subset S^{2n-1}$ a countable set. Then the set
\[
\{\Phi^X(z,t):z\in F,\, t\in \mathbb{C},\,\textup{Re}\,t=0\}\subset S^{2n-1}
\]
is a $\lambda$-circular $F_{\sigma}$ set containing $F$. So Theorem \ref{normalsuspension} yields the following
\begin{corollary}\label{countablenonnormal}
If $F\subset S^{2n-1}$ is countable, then the suspension $S^X_0(F)$ is always nonnormal.
\end{corollary}
 The method of Sadullaev \cite{Sadullaev22} also yields the following estimate on the region of convergence of $S$. Recall the notation in \textup{(}\ref{z_lambda}\textup{)}.
\begin{theorem}\label{Sadullaev}
  Given the assumptions in Theorem \ref{normalsuspension}, $S$ converges uniformly on an open neighborhood
	\[
	\Omega':=\{z\in \mathbb{C}^n:\|z\|^{\frac{\textup{min}(\lambda)}{\textup{max}(\lambda)}}\cdot \big\{\Phi^{\ast}_{S^X_0(F)}(z_{\lambda})\big\}<1\}
	\]
	of the origin. 
\end{theorem}
 Compare Theorem \ref{Sadullaev} with Theorem 3.1 in \cite{Sadullaev22}.
\begin{proof}
We will use the same notations as in the proof of Theorem \ref{normalsuspension}. Then in particular, $|q_m(z)|\leq k$ holds for each $m\geq 0$ and $z\in L_{k,\ell}$. So by Definition \ref{pluricomplexdef}, we have 
\begin{align*}
|q_m(z)|\leq \|q_m\|_{L_{k,\ell}}\cdot \{\Phi^{\ast}_{L_{k,\ell}}(z)\}^{\textup{deg}\,q_m}\leq k\cdot \{\Phi^{\ast}_{L_{k,\ell}}(z)\}^{\textup{deg}\,q_m}
\end{align*}
for each $k,\ell,m$ and $z\in \mathbb{C}^n$. Note also that $\Phi^{\ast}_{L_{k,\ell}}\geq 1$ on $\mathbb{C}^n$ by Definition \ref{pluricomplexdef}. Therefore,
\[
|q_m(z)|=\|z\|^{\frac{\rho_m}{\textup{max}(\lambda)}}\cdot |q_m(z_{\lambda})|\leq k\cdot \|z\|^{\frac{\rho_m}{\textup{max}(\lambda)}}\cdot \{\Phi^{\ast}_{L_{k,\ell}}(z_{\lambda})\}^{\frac{\rho_m}{\textup{min}(\lambda)}}
\]
so that
\[
\limsup_{m\to \infty}|q_m(z)|^{\frac{1}{\rho_m}}\leq \|z\|^{\frac{1}{\textup{max}(\lambda)}}\cdot \{\Phi^{\ast}_{L_{k,\ell}}(z_{\lambda})\}^{\frac{1}{\textup{min}(\lambda)}}
\]
for each $k,\ell \geq 1$ and $z\in \mathbb{C}^n-\{0\}$. Then it follows from Theorem 2.9 in \cite{Siciak90} and (\ref{union}) that
\begin{align}\label{limsupest}
	\limsup_{m\to \infty}|q_m(z)|^{\frac{1}{\rho_m}}\leq \|z\|^{\frac{1}{\textup{max}(\lambda)}}\cdot \{\Phi^{\ast}_{S^X_0(F)}(z_{\lambda})\}^{\frac{1}{\textup{min}(\lambda)}}
\end{align}
for each $z\in \mathbb{C}^n-\{0\}$. Since $S^X_0(F)$ is nonpluripolar, we have $\textup{log}\,\Phi^{\ast}_{S^X_0(F)}\in \mathcal{L}_n$ by Theorem 2.9 in \cite{Siciak90}. So there is a constant $C\geq 0$ such that
\[
\Phi^{\ast}_{S^X_0(F)}(z)\leq C\cdot (1+\|z\|)~\text{on}~\mathbb{C}^n.
\] 
Note that the function 
\begin{align*}
u(z):&=\|z\|^{\frac{1}{\textup{max}(\lambda)}}\cdot \{\Phi^{\ast}_{S^X_0(F)}(z_{\lambda})\}^{\frac{1}{\textup{min}(\lambda)}}\\
&= \textup{exp}\,\bigg(\frac{1}{\textup{max}(\lambda)}\cdot\textup{log}\,\|z\|+\frac{1}{\textup{min}(\lambda)}\cdot \textup{log}\,\Phi^{\ast}_{S^X_0(F)}(z_{\lambda})\bigg)
\end{align*}
is upper-semicontinuous on $\mathbb{C}^n-\{0\}$. Since 
\[
\limsup_{z\to 0}u(z)\leq C^{\frac{1}{\textup{min}(\lambda)}}\cdot \limsup_{z\to 0}\,\{\|z\|^{\frac{1}{\textup{max}(\lambda)}}\cdot(1+\|z_{\lambda}\|)^{\frac{1}{\textup{min}(\lambda)}}\}=0,
\] 
$u$ can be extended to an upper-semicontinuous function $\tilde{u}$ on $\mathbb{C}^n$ with $\tilde{u}(0)=0$. Therefore, $\Omega'\subset \mathbb{C}^n$ is an open neighborhood of the origin and by (\ref{limsupest}), $S$ defines a holomorphic function on $\Omega'$.
\end{proof}

The following example shows that the normality of $S^X_0(F)$ depends on both $F$ and $X$.
\begin{example}
	\normalfont
	Fix positive integers $m,n$. Let $X_m$ be the vector field on $\mathbb{C}^2$ with eigenvalues $\lambda_m:=(1,m)$ and define
	\[
	F_n:=\bigg\{\bigg(\frac{e^{i\theta}}{\sqrt{2}},\frac{e^{in\theta}}{\sqrt{2}}\bigg)\in S^3: \theta \in [0,2\pi]\bigg\}.
	\]
	Note that $(F_n)^{'}_{\lambda_m}=(F_n)^{'}_{\lambda_m,1}=(F_n)^{'}_{\lambda_m,2}$ since the components of $\lambda_m$ are positive integers. So we have
	\begin{align*}
		(F_n)^{'}_{\lambda_m}&=\{(\sqrt{2})^{m-1}e^{i(n-m)\theta}: \theta \in [0,2\pi]\}\\
		&=
		\begin{cases}
			\{(\sqrt{2})^{m-1}\}~\text{if}~m=n, \\
			\{z\in \mathbb{C}: |z|=(\sqrt{2})^{m-1}\}~\text{if}~m\neq n.
		\end{cases}
	\end{align*}
	Then by Theorem \ref{normalsuspension}, $S^{X_m}_0(F_n)$ is normal if, and only if, $m\neq n$.
\end{example}

\section{Holomorphic extension along flows}\label{Sect ext}

Once $f:B^n\to \mathbb{C}$ in Theorem \ref{main theorem} is shown to be holomorphic on $\Omega$, it is natural to ask whether the function can be extended holomorphically along the suspension. In \cite{KPS09} and \cite{JKS16}, the authors use the rectification theorem \cite{IY07} and Lemma \ref{Hartogslemma} to show that $f$ extends to a holomorphic function on the union of $\Omega$ and all maximal integral curves of $X$ when the given suspension is $S^X_0(S^{2n-1})=B^n$. But as suspensions in our case may not be open, we need the following generalization of Lemma \ref{Hartogslemma} by Shiffman.

\begin{proposition}[\cite{Shiffman89}]\label{ShiffmanHartogs}
Let $U$ be an open subset of $\mathbb{C}^m$ and $P^n(0;r)\subset \mathbb{C}^n$ the polydisc of multi-radius $(r,\dots,r),\,r<1$. Suppose that a set $E\subset U$ is locally $L$-regular at $z_0\in E$. If $f: U \times P^n(0;r)  \to \mathbb{C}$ is holomorphic and the function $f_{z}:w\in P^n(0;r)\to f(z,w)$ extends to a holomorphic function on $P^n:=P^n(0;1)$ for each $z\in E$, then there exists an open neighborhood $U_0=U_0(r,E,U)\subset \mathbb{C}^m$ of $z_0$ such that $f$ on $U_0\times P^n(0;r)$ extends to a holomorphic function on $U_0\times P^n$.
\end{proposition}
 Since it may not be clear from the proof in \cite{Shiffman89} that $U_0$ can be chosen to be dependent only on $r,E,$ and $U$, we give a slightly modified version of the proof.
\begin{proof}
 Let $0<s<r<R<1$ and $V:={U\times P^n(0;s)}$. Assume that $\|f\|_{V}=K<\infty$ for some $K\geq 0$ by shrinking $U$, if necessary. By the given assumption, one can write
\begin{equation}\label{series}
	f:=\sum_{m=0}^{\infty}g^m~\text{on}~U\times P^n(0;r),
\end{equation}
where each $g^m$ is holomorphic on $U\times \mathbb{C}^n$ and $g^m_z(w):=g^m(z,w)\in \mathbb{C}[w_1,\dots,w_n]$ is a homogeneous polynomial of degree $m$ for each $z\in U$. By the Cauchy estimate, we have
\[
\|g^m_z\|_{P^n}\leq Ks^{-m}~\text{for each}~ m\geq 0,\,z\in U.
\]
 Since $f_z$ extends to a holomorphic function on $P^n$ for each $z\in E$, we also have
\[
\|g^m_{z}\|_{P^n}\cdot R^m\to 0~\text{as}~m\to \infty
\]
for any $z\in E$. Define 
\[
u_m(z):=\sup\limits_{\ell\geq m}\frac{1}{\ell}\,\textup{log}\,\|g^{\ell}_{z}\|_{P^n}+\textup{log}\,R\in \textup{PSH}(U)
\]
for each $m\geq 0$. Then by the previous arguments, we have
\begin{enumerate}
	\setlength\itemsep{0.1em}
	\item $u_m(z)\leq \textup{log}\,\frac{R}{s}$ for each $m\geq 0,\,z\in U$, and
	\item $\limsup\limits_{m\to \infty}u_m(z)\leq -3\,\textup{log}\,2~ \text{for any}~ z\in E$.
\end{enumerate}
Let $h_{EU}$ be the $\textit{relative extremal function}$ for $E$ in $U$ defined as
\[
h_{EU}(z):=\textup{sup}\,\{u(z):u\in \textup{PSH}(U): u\leq 0~\text{on}~E,\, u\leq 1~\text{on}~U \}
\]
for each $z\in U$. Then it is known that $h^{\ast}_{EU}(z_0)=0$ and $h^{\ast}_{EU}\in \textup{PSH}(U)$. So the set
\[
 U_1:=\bigg\{z\in U:h^{\ast}_{EU}(z)< \textup{log}\,2\cdot \bigg(\textup{log}\,\frac{R}{s}\bigg)^{-1}\bigg\}
\]
is an open neighborhood of $z_0$ dependent only on $r,E,$ and $U$. Now one can proceed as in the proof of Lemma 2 in \cite{Shiffman89} and show that there exists an open neighborhood $U_0\subset\subset U_1$ of $z_0$ and a positive integer $N$ such that $u_m(z)\leq -\textup{log}\,2$ if $z\in U_0,\,m\geq N$.
This reduces to
\[
|g^m(z,w)|\leq \bigg(\frac{1}{2}\bigg)^m~\text{for each}~(z,w)\in U_0\times P^n,\, m\geq N.
\]
 Therefore, we conclude from the Weierstrass $M$-test that the series in (\ref{series}) defines a holomorphic function on $U_0\times P^n$.
\end{proof}

\begin{proposition}\label{extension}
Let $f:B^n\to \mathbb{C}$ be a function holomorphic on $B^n(0;r)$ for some $r\in (0,1)$. If $f$ is holomorphic along a regular suspension $S^{X}_0(F)$, then there exists a neighborhood $U=U(F,X)\subset S^{2n-1}$ of a generator $z_0\in \bar{F}$ of the regular leaf such that $f|_{B^n(0;r)}$ extends to a holomorphic function on $B^n(0;r)\cup S^X_0(U)$.
\end{proposition}

\begin{proof}
By applying a unitary transformation if necessary, we may assume that $z_1\neq 0$ if $z=(z_1,\dots,z_n)\in L_{z_0}$. Fix $w\in L_{z_0}$ with $r/2<\|w\|<1$. By Remark \ref{uniformIFT}, there exist numbers $s,\epsilon>0$ independent of the choice of $w$ such that the following hold up to a change of local holomorphic coordinate system at $w$: 
\begin{enumerate}
\setlength\itemsep{0.1em}
\item $P^n(w;2\epsilon)\cap L_{z_0}=\{(t,0,\dots,0): t\in \mathbb{H},\,\textup{Re}\,t< s\}.$
\item Each flow curve of $X$ in $P^n(w;\epsilon)$ is parametrized as $t \to (t,tz'_1,\dots,tz'_{n-1})$ for some $(z'_1,\dots, z'_{n-1})\in F'_{\lambda}$.
\item $f$ is holomorphic along the set of lines in (2).
\end{enumerate}
Since $L_{z_0}$ is a regular leaf, it follows from Remark \ref{L-regular} and Theorem \ref{regularnonpp} that any open neighborhood of $w$ intersects another regular leaf of $S^X_0(F)$. So by Proposition \ref{ShiffmanHartogs}, there exists a number $\delta=\delta(F,X)>0$ such that any function $f:B^n\to \mathbb{C}$ that is (1) holomorphic on $B^n(0;\|w\|),$ and (2) holomorphic along $S^X_0(F)$ extends to a holomorphic function on $B^n(0;\|w\|)\cup B^n(w;\delta)$.

 Let $f:B^n\to \mathbb{C}$ be the given function and choose a number $d>0$ such that 
\[
\|\Phi^{X}(z_0,t)\|=\frac{3r}{4}
\]
holds for any $t\in \mathbb{C}$ with $\textup{Re}\,t=d$. Define
\[
A:=\{s\in (0,d]: f~\text{is holomorphic on}~B^n(\Phi^X(z_0,t);\delta)~\forall t\in \mathbb{C}~\text{with}~\textup{Re}\,t=s\}.
\]
Then $d\in A$ by the given assumption so $A\neq \emptyset$. Suppose that $s_0:=\textup{inf}\,A\neq 0$ and fix $z\in B^n$ with $\|z\|=s_0$. By the preceding arguments, $f$ is holomorphic on $B^n(z;\delta)$. So we have $s_0>\textup{inf}\,A$ which is a contradiction. Therefore, $\textup{inf}\,A= 0$ and this completes the proof.
\end{proof}
\textit{Proof of Theorem \ref{main theorem}}. Let 
$f\colon B^n\to \mathbb{C} $ be a function that is smooth at the 
origin and holomorphic along a Forelli suspension $S^X_0(F)$. Then the 
formal Taylor series $S_f$ is of holomorphic type by Theorem \ref{formal Forelli suspension theorem}. Note also that $S_f$ converges uniformly on $\Omega$ by Theorem \ref{normalsuspension}. Now $f\equiv S_f$ is holomorphic on $\Omega$ and moreover, Proposition \ref{extension} implies that there exists an open neighborhood $U\subset S^{2n-1}$ of the regular leaf of $S^X_0(F)$ such that $f|_{\Omega}$ extends to a holomorphic function defined on $\Omega\cup S^X_0(U)$. Then the conclusion follows from (\ref{domofholoext}) and Theorem \ref{smallestdomainofholo}. \hfill $\Box$

\section{Examples of Suspensions}\label{Sect example}

In this section, we follow the ideas in \cite{Cho22} to construct several examples of suspensions.
\begin{example}\label{rational suspension}
	\normalfont 
	Fix a vector field $(X,\lambda)$ on $\mathbb{C}^n$. As every point of a nonempty open subset $U \subset S^{2n-1}$ generates a regular leaf, $S^X_0(U)$ is a Forelli suspension by Corollary $\ref{opensuspension}$. Define 
	\[
	F:=\{(z_1,\dots,z_n)\in U: z_i\in \mathbb{Q}~\forall i\in \{1,...,n\}\}.
	\]
	Then by Corollary \ref{countablenonnormal}, $S^X_0(F)$ is not normal. We also conclude from Corollary \ref{opensuspension} that $S^X_0(F)$ is a dense formal Forelli suspension as $S^X_0(\bar{F})=S^X_0(U)$.
\end{example}

In the following, we identify $\mathbb{R}^{2n-1}$ with the set $\{(z_1,\dots,z_n)\in \mathbb{C}^n:  \textup{Im}\,z_1=0\}.$ 
\begin{example}\label{nowhere dense formal Forelli suspension}
	\normalfont 
	Fix a vector field $(X,\lambda)$ on $\mathbb{C}^2$ with eigenvalues $\lambda=(1,\lambda_2),\,\lambda_2>0$.	Let $\{r_{k}\}$, $\{s_\ell\}\subset \mathbb{R}$ be two sequences that decreases from $\frac{\pi}{4}$ to $0$, increases from $\frac{\pi}{4}$ to $\frac{\pi}{2}$, respectively. For each positive integer $\ell$, define
	\begin{align*}
			F_{\ell}:=\{(\textup{cos}\,r_{k},\,e^{is_\ell}\textup{sin}\,r_{k}) \in   \mathbb{R}^3\cap S^3:k~ \text{is a positive integer}\}
	\end{align*}
and let $F:=\bigcup_{\ell=1}^{\infty}{F}_{\ell}.$ Then by Corollary \ref{countablenonnormal}, $S^X_0(F)$ is not normal. We prove that $S^X_0(F)$ is a nowhere dense formal Forelli suspension by showing that $v:=(1,0)\in \bar{F}$ generates a nonsparse leaf. 
	
	Choose an open neighborhood $U$ of $v$ in $S^3$ and suppose that $S^X_0(\bar{F}\cap U)\subset Z(q)$ for some $q\in \mathcal{H}_{\lambda}$ with $\textup{bideg}\,q=(d_1,d_2)$, $d_2\neq 0$. Then $q$ can be written as
	
	\begin{equation}\label{sum}
		q(z_1,z_2,\bar{z}_1,\bar{z}_2)=\sum_{\substack{\alpha+\lambda_2 \beta=d_1 \\ \gamma+\lambda_2 \delta=d_2\\ }}C_{\alpha\beta}^{\gamma\delta}\cdot z_1^{\alpha}z_2^{\beta}\bar{z}^{\gamma}_1\bar{z}^{\delta}_2,
	\end{equation}
	where $\{C_{\alpha\beta}^{\gamma\delta}\}$ is a finite set of complex numbers with $0\leq \alpha,\beta,\gamma,\delta\leq N$. Now we are to show that $q\equiv 0$ on $\mathbb{C}^2$. As $v$ is a limit point of each $F_k$, there exists a positive integer $M$ such that
	\[
	(\textup{cos}\,r_{k},e^{is_\ell}\textup{sin}\,r_{k})\in S^{X}_0(\bar{F}\cap U)\subset Z(q)
	\] 
	if $k,\ell\geq M.$ Then
	\begin{align}	\label{Equation}
		0&=q(\textup{cos}\,r_{k},e^{is_\ell}\textup{sin}\,r_{k}) \nonumber \\
		&=\sum_{\substack{\alpha+\lambda_2 \beta=d_1 \\ \gamma+\lambda_2 \delta=d_2\\ }} \big\{C_{\alpha\beta}^{\gamma\delta}\cdot (\textup{cos}\,r_{k})^{\alpha+\gamma}(\textup{sin}\,r_{k})^{\beta+\delta}e^{is_\ell(\beta-\delta)}\big\}.
	\end{align}
	Fix nonnegative integers $m,r$. Note that $(\ref{Equation})$ is equivalent to the following equation
	\begin{equation*}
		g(z)=\sum_{n=-N}^{N}P_n(\textup{cos}\,r_{k},\, \textup{sin}\,r_{k})\,z^n=0~\forall z\in \{e^{is_\ell}\},
	\end{equation*}
	where 
	\[
	P_n(x,y)=\sum_{\substack{\alpha+\lambda_2 \beta=d_1 \\ \gamma+\lambda_2 \delta=d_2\\ \beta-\delta =n }}  C_{\alpha\beta}^{\gamma\delta}\cdot x^{\alpha+\gamma}y^{\beta+\delta}
	\]
	is a polynomial in real variables $x,y$. As $g$ is holomorphic on $\mathbb{C}-\{0\}$, it follows from the identity theorem that $P_n(1,t_k)= 0$ for each integer $k$ and $n$, where
	\[
	t_k:=\frac{\textup{sin}\,r_{k}}{\textup{cos}^{\lambda_2}\,r_{k}}.
	\] 
	Then finally, the coefficient of each monomial in $P_n(1,t)$ is zero since $P_n(1,t)\in \mathbb{C}[t]$ is a finite polynomial. Therefore $\sum C_{\alpha\beta}^{\gamma\delta} = 0,$ where the sum is taken over all quadruple $(\alpha,\beta,\gamma,\delta)$ satisfying
	\begin{equation}\label{linear equation}
		\begin{cases}
			\alpha+\lambda_2\beta=d_1 \\
			\gamma+\lambda_2\delta =d_2 \\
			\beta+\delta = r\\	
			\beta-\delta = m.	
		\end{cases}
	\end{equation}
	Note that (\ref{linear equation}) always has a unique solution. So we have $C_{\alpha\beta}^{\gamma\delta}=0$ for any quadruple $(\alpha,\beta,\gamma,\delta)$ appearing in (\ref{sum}). Then $q\equiv 0$ on $\mathbb{C}^2$ as desired.

	This construction can be generalized to higher dimensions. Fix a vector field $(X,\lambda)$ on $\mathbb{C}^n$ with $\lambda=(1,\lambda_2,\dots,\lambda_n)$. Let $x_1(\theta)=\textup{cos}\,\theta,$ $x_2(\theta)=\textup{sin}\,\theta$ be the parametrization of $S^1$ and define a parametrization of $S^{n+1}$ inductively as
	\begin{equation*}
		\begin{cases}
			x_i(\theta_1,\dots,\theta_{n},\theta_{n+1}) =x_i(\theta_1,\dots,\theta_{n})\cdot \textup{cos}\,\theta_{n+1} ~\text{for}~1\leq i \leq n+1,\\
			x_{n+2}(\theta_1,\dots,\theta_{n},\theta_{n+1}) = \textup{sin}\,\theta_{n+1},
		\end{cases}
	\end{equation*}
	where $\{x_i(\theta_1,\dots,\theta_n):1\leq i\leq n+1\}$ is the parametrization of $S^{n}$ chosen in the previous induction step. Fix two $(n-1)$-tuples $k:=(k_1,\dots,k_{n-1}),\,\ell:=(\ell_1,\dots,\ell_{n-1})$ of positive integers and define
	\begin{gather*}
		x_i(k):=x_i(r_{k_1},\dots,r_{k_{n-1}})~\forall i\in \{1,\dots,n\},\\
		F^n_{k\ell}:=	\{(x_1(k),e^{is_{\ell_1}}x_2(k),\dots,e^{is_{\ell_{n-1}}}x_{n}(k)) \in   \mathbb{R}^{2n-1}\cap S^{2n-1}\},
	\end{gather*}
	and $F^n:=\bigcup_{k,\ell}{F}^n_{k\ell}.$ Then one can proceed as before to show that $S^{X}_0(F^n)$ is not normal and $v_n=(1,0,\dots,0)\in S^{2n-1}$ generates a nonsparse leaf of $S^{X}_0(F^n)$ for each positive integer $n$. Therefore, $S^{X}_0(F^n)$ is a nowhere dense formal Forelli suspension for any vector field $X$ on $\mathbb{C}^n$.
\end{example}

\begin{example}\label{nowhere dense Forelli suspension}
	\normalfont
	Let $\{s_{\ell}\}$ and $(X,\lambda),\,\lambda=(1,\lambda_2)$ be the same as in Example \ref{nowhere dense formal Forelli suspension}. For each positive integer ${\ell}$, define a copy of $S^1$ in $S^3$ as
	\begin{align*}
		G_{\ell}:=\{(x,e^{is_{\ell}}y)\in \mathbb{R}^3&\cap S^{3}:\,x,y\in \mathbb{R}\}.
	\end{align*}
	Note that $(G_{\ell})'_{\lambda,1}\subset \mathbb{C}$ is biholomorphic to the real line $\mathbb{R}=\{z_1\in \mathbb{C}:\,\textup{Im}\,z_1=0\}$. By applying the Phragm$\acute{\textup{e}}$n-Lindel$\ddot{\textup{o}}$f principle for subharmonic functions (see p.33 of \cite{Ransford95}), one can check that $V^{\ast}_{\mathbb{R}}(z)=0$ for any $z\in \mathbb{C}$. Therefore, every point of $G_{\ell}$ generates a regular leaf and in particular, each $S^X_0(G_{\ell})$ is normal. Note that the suspension $S^X_0(G)$ generated by $G:=\bigcup_{\ell=1}^{\infty}G_\ell$ is a nowhere dense Forelli suspension as it contains the normal suspension $S^{X}_0(G_1)$ and the formal Forelli suspension $S^{X}_0(F)$ constructed in Example \ref{nowhere dense formal Forelli suspension}. Note also that $v=(1,0)\in G$ generates a regular leaf and a nonsparse leaf of $S^{X}_0(G)$.
	
	This construction can also be generalized to higher dimensions. Fix a vector field $(X,\lambda)$ on $\mathbb{C}^n$ and for each positive integer $\ell$, define
	\[
	G^n_{\ell}:=\{(x,z_2,\dots,z_{n-1},e^{is_{\ell}}y)\in \mathbb{R}^{2n-1}\cap S^{2n-1}:\,x,y\in \mathbb{R},\,z_i\in \mathbb{C}~\forall i\}.
	\]
	Then each $(G^n_{\ell})'_{\lambda,1}=\mathbb{C}^{n-2}\times \mathbb{R}\subset \mathbb{C}^{n-1}$ is $L$-regular at every point of itself. So $S^{X}_0(G^n_{\ell})$ is normal. Note that the suspension generated by $G^n:=\bigcup_{\ell=1}^{\infty}G^n_\ell$ contains the formal Forelli suspension $S^{X}_0(F^n)$ constructed in Example $\ref{nowhere dense formal Forelli suspension}$. Therefore, $S^{X}_0(G^n)$ is a nowhere dense Forelli suspension for any vector field $X$ on $\mathbb{C}^n$.
\end{example}

\vspace{50pt}

Ye-Won Luke Cho (\texttt{ww123hh@pusan.ac.kr}) 

\medskip

Department of Mathematics,

Pusan National University, 

Busan 46241, The Republic of Korea.

\end{document}